\documentclass[11pt]{amsart}

\usepackage{latexsym}
\usepackage{amssymb}
\usepackage{amsmath}
\usepackage{xcolor}
\usepackage{tikz-cd}

\usepackage[italian, english]{babel}
\usepackage{url}

\newtheorem{theorem}{Theorem}[section]

\newtheorem{proposition}[theorem]{Proposition}
\newtheorem{corollary}[theorem]{Corollary}
\newtheorem{question}[theorem]{Question}

\theoremstyle{definition}
\newtheorem{definition}[theorem]{Definition}

\theoremstyle{remark}
\newtheorem{remark}[theorem]{Remark}

\theoremstyle{example}
\newtheorem{example}[theorem]{Example}

\def\N{\mathbb{N}}
\def\Z{\mathbb{Z}}
\def\R{\mathbb{R}}
\def\XX{\mathbb{X}}
\def\YY{\mathbb{Y}}

\def\XX{\mathbf{X}}
\def\YY{\mathbf{Y}}

\def\s{\mathfrak{s}}

\def\i{\mathfrak{i}}

\def\dd{\mathfrak{d}}

\def\hN{{{}^*\N}}

\def\hR{{{}^*\R}}
\def\hA{{{}^*\!A}}
\def\hB{{{}^*\!B}}

\def\hX{{{}^*\!X}}

\def\hY{{{}^*Y}}

\def\A{\mathcal{A}}

\def\G{\mathcal{G}}

\def\M{\mathcal{M}}
\def\PP{\mathcal{P}}

\def\U{\mathcal{U}}

\def\UU{\mathfrak{U}}

\def\VV{\mathbb{V}}

\def\ueq{{\,{\sim}_{{}_{\!\!\!\!\! u}}\;}}

\def\res{\!\upharpoonright\!}

\makeatletter
\@namedef{subjclassname@2020}{%
  \textup{2020} Mathematics Subject Classification}
\makeatother

\begin{document}

\title[Foundations of iterated star maps]
{Foundations of iterated star maps
\\
and their use in combinatorics}

\author{Mauro Di Nasso}

\author{Renling Jin}

\address{Dipartimento di Matematica\\
Universit\`a di Pisa, Italy}

\email{mauro.di.nasso@unipi.it}

\address{Department of Mathematics\\
College of Charleston, U.S.A.}

\email{jinr@charleston.edu}



\begin{abstract}
We develop a framework for nonstandard analysis
that gives foundations to the interplay between external and internal
iterations of the star map, and we present a few examples to show the strength
and flexibility of such a nonstandard technique
for applications in combinatorial number theory.
\end{abstract}

\maketitle

\section*{Introduction}

In the past decade or so, iterated star maps in nonstandard analysis
have been successfully used to prove new results
in Ramsey theory and combinatorics of numbers (see, \emph{e.g.}, \cite{dn1,lu,dl,blm}).
To the authors' knowledge, a first use of double nonstandard extensions
in applications goes back to the 80s years of the last century
with the work by V.A. Molchanov \cite{mo} in topology
(see also \cite{gy} for a related research).
However, that approach did not find further developments,
and no more use of similar ideas appears to have been made
until 2010, when the first named author started considering
iterated star maps as a convenient tool
to formalize algebra of ultrafilters in applications to combinatorics (see \cite{dn-ln,dn1,dn2}).
The foundations of such iterated star maps was then studied
by L. Luperi Baglini in his Ph.D. thesis \cite{luPhD},
and applications in Ramsey Theory and combinatorics
initiated with the works \cite{dn1} and \cite{lu}.
Recently, iterated bounded elementary embeddings were used
by the second named author in \cite{jin} to simplify the combinatorial arguments
used in the original proof of Szemer\'{e}di's Theorem.
The new idea there is to utilize various versions of iterated
star maps from a nonstandard universe
to another nonstandard universe, based on whether each involved step
of the iteration is done ``internally" or ``externally".

\smallskip
The construction of iterated ultraproducts was first introduced by
H. Gaifman \cite{ga} in the 60s years of the last century,
and is now a relevant topic both in set theory, starting from the
work by K. Kunen \cite{ku} (see \emph{e.g.} \cite{ms}), and in model theory
(see \emph{e.g.} \cite[\S 6.5]{ck}). Also in the context of nonstandard analysis the idea
of iterating nonstandard embeddings, also in their ``internal" version,
has been already considered (see the foundational
work by K. Hrb\'{a}\v{c}ek in the
general context of nonstandard set theory, \emph{e.g.} \cite{hr}).
However, the above studies do not seem to consider, to the knowledge of these authors,
compositions in which ``internal'' and ``external'' star maps are alternated.
Such compositions are possible in a natural way when
the considered nonstandard embedding $*:\VV\to\VV$ is an embedding
from a universe into itself.

\smallskip
The first explicit construction of a star map $*:\VV(\XX)\to\VV(\XX)$ from a superstructure
over a set of atoms $\XX$ into itself was presented by V. Benci in \cite{be}.
In the general setting of set theory, the possibility of taking a bounded elementary extension
of the whole universe into itself as a foundation for nonstandard analysis,
was first considered by D. Ballard and K. Hrb\'{a}\v{c}ek in \cite{bh}.\footnote
{~Such bounded elementary embeddings can be constructed in
a set theory where one drops the axiom of foundation and replace it
with a suitable axiom that guarantees the existence of a Mostowski's transitive
collapse for every extensional structures.}
That idea was then adopted and studied by the first named author, and
an axiomatic nonstandard set theory that incorporates a star map
defined on the entire well-founded universe was introduced in \cite{dn00,dn0}.

\smallskip
Our goal here is to systematically develop a framework for nonstandard analysis
where alternate iterations of the star map with ``internal" versions of it are possible,
and to demonstrate its strength and flexibility in applications.
To this end, we will present new nonstandard proofs of some theorems
in arithmetic Ramsey theory where this technique is used.

\smallskip
The paper is organized in the following sections. In \S\ref{sec-basics},
the foundations for the use of iterated star maps in nonstandard analysis are recalled,
within the most commonly used framework of superstructures.
In \S\ref{sec-bistarmap}, the \emph{internal star map} $\circledast$ is introduced, and some basic properties
of two-step iteration of the star maps are proven.
In \S\ref{sec-internalstarmaps} the $n$-th internal star maps
are introduced, which correspond to the internal
version of iterated star maps, and the property
of \emph{transfer} is proved for them.
The possible compositions of iterated star maps and
$n$-th internal star maps are then studied,
and a representation theorem is proved.
\S \ref{sec-bee} is focused on the possible bounded elementary embedding
between internal universes of level $n$, and their properties.
In the next \S\ref{sec-combinatorics}, a few basic results
that connect iterated star maps with Ramsey properties are presented.
Finally, in \S\ref{sec-applications}, three results in
combinatorics of numbers, namely the classic Ramsey's Theorem,
the multidimensional van der Waerden's Theorem,
and the partition regularity of exponential triples, are proved,
so to illustrate the nonstandard technique presented here.
In the closing \S\ref{question}, some questions are asked.

\medskip
\section{Models of nonstandard analysis}\label{sec-basics}

We will follow the most used approach to nonstandard analysis
based on superstructures. Let us recall here the basic notions.\footnote
{~Full details about the
superstructure approach to nonstandard analysis are found in \cite[\S 4.4]{ck};
see also \cite[Ch.15]{keisler}.}

\begin{definition}
The \emph{superstructure} over a set of atoms (or urelements) $\XX$
is the union $\VV(\XX)=\bigcup_{n=0}^\infty V_n(\XX)$ where
$V_0(\XX)=\XX$, and inductively $V_{n+1}(\XX)=V_n(\XX)\cup\mathcal{P}(V_n(\XX))$.

It is assumed that $\XX\supseteq\R$, i.e., $\XX$ contains the set of real numbers.
\end{definition}

We remark that any superstructure contains all mathematical objects of interest
and is closed under all fundamental set theoretic operations.
So, it is safe to say that a superstructure is a
``mathematical universe", in the sense that virtually all mathematical constructions
that one performs in the practice, with the exception of some general constructions
in set theory and topology, are carried within it.

It is easily seen that if $y\in x\in V_n(\XX)$ then $y\in V_{n-1}(\XX)$,
and therefore the superstructure $\VV(\XX)$ is transitive.\footnote
{~Recall that a set $T$ is \emph{transitive} if
elements of elements of $T$ are elements of $T$, \emph{i.e.}
$a\in b\in T\Rightarrow a\in T$.}

\begin{definition}
A (superstructure) \emph{model of nonstandard analysis}
(or a \emph{nonstandard model}) is a triple
$\langle \VV(\XX),\VV(\YY), *\rangle$ where:
\begin{itemize}
\item
The \emph{star-map} $*:\VV(\XX)\to\VV(\YY)$ (or \emph{nonstandard embedding})
is a map whose domain is a universe $\VV(\XX)$, named \emph{standard universe},
and that takes values in a universe $\VV(\YY)$, named \emph{nonstandard universe};
\item
${}^*\XX=\YY$;
\item
The \emph{transfer principle} holds:
For every bounded quantifier formula $\varphi(x_1,\ldots,x_n)$
and for all elements $A_1,\ldots,A_k\in\VV(\XX)$:
$$\varphi(A_1,\ldots,A_n)\Longleftrightarrow\varphi(\hA_1,\ldots,\hA_n).$$
\end{itemize}

Besides, it is assumed that ${}^*r=r$ for all $r\in\R$, and
that the non-triviality condition $\hN\ne\N$ holds.

If $A\in\VV(\XX)$, its image $\hA$ under the star-map
is called the \emph{nonstandard extension} (or the \emph{star extension}, or
the \emph{hyper-extension}) of $A$.\footnote
{~To simplify the notation, it is customary to write $\hA$ instead of $*(A)$.}
\end{definition}

Recall that a \emph{bounded quantifier formula}
is a first-order formula in the language of set theory
where all quantifiers appear in the bounded forms $\forall x\in y$ or $\exists x\in y$.
In other words, the \emph{transfer principle} states that
the star map $*$ is a \emph{bounded elementary embedding}.


As a consequence of the \emph{transfer principle}, the nonstandard extension $\hA\in\VV(\YY)$
of a mathematical object $A\in\VV(\XX)$ is a sort of ``weakly isomorphic" copy of $A$,
in that $\hA$ and $A$ satisfy the same ``elementary properties",
as formalized by bounded quantifier formulas.
For instance, the nonstandard extension $\hR\supsetneq\R$ of the real numbers
will be an ordered field that properly extends the real line, and hence it will
contain infinitesimal and infinite numbers.


\begin{definition}
An element $x\in\VV(\YY)$ is \emph{internal} if it belongs to some
hyper-extension $x\in\hA$. The \emph{internal universe}
${}^*\VV(\XX)\subseteq\VV(\YY)$ is the class of internal elements:
$${}^*\VV(\XX):=\bigcup_{X\in\VV(\XX)}{}^*\!X$$
Elements of the universe that are not internal
are called \emph{external}.
\end{definition}

The first basic examples of external objects are the sets $\N$ and $\R$
of natural numbers and real numbers. It is worth stressing the fact
that in the practice of nonstandard analysis, external objects are
of central importance. Probably the most relevant example is
given by the \emph{Loeb measure} construction (see \emph{e.g} \cite{cu}).

It is easily seen that the internal universe $^*\VV(\XX)$
is the \emph{transitive closure} of $\text{range}(*)=\{{}^*\!X\mid X\in\VV(\XX)\}$.\footnote
{~Recall that the transitive closure of a set $A$ is the smallest transitive set $T\supseteq A$.
Such a set can be defined as $T=\bigcup_{n\ge 0}A_n$ where $A_0=A$ and,
inductively, $A_{n+1}=\bigcup_{a\in A_n}a$.}

Sometimes, elements of $\text{range}(*)$
are named \emph{internal-standard}, because they are
images under the map $*$ of standard elements, and they are internal.

By \emph{transfer}, it is shown that a subset $B\subseteq\hX$
of an hyper-extension is internal if and only if
$B\in {}^*\PP(X)$ belongs to the hyper-extension of the corresponding powerset.

\smallskip
Useful additional properties that a model of nonstandard analysis may have are
the following, where $\kappa$ is a fixed infinite cardinal.
\begin{itemize}
\item
\emph{$\kappa$-enlargement property}.
Let $\G$ be a ``bounded" family of standard sets, \emph{i.e.}
$\G\subseteq A$ for some $A\in\VV(\XX)$.
If $|\G|<\kappa$ and
$\G$ satisfies the \emph{finite intersection property}, then $\bigcap_{G\in\G}{}^*G\ne\emptyset$.
\item
\emph{$\kappa$-saturation property}.
Let $\G$ be a ``bounded" family of internal sets, \emph{i.e.}
$\G\subseteq{}^*\!A\cap{}^*\VV(\XX)$ for some $A$.
If $|\G|<\kappa$ and
$\G$ satisfies the \emph{finite intersection property}, then $\bigcap_{B\in\G}B\ne\emptyset$.
\end{itemize}

If $\G\subseteq A$ is a family of standard sets with the finite intersection property,
then also the family of internal sets
$*[\G]=\{\hB\mid B\in\G\}$ has the finite intersection property, and so it is clear
that $\kappa$-saturation implies $\kappa$-enlargement
(but the converse implication does not hold).


\subsection{Star maps from a universe into itself}

\

\smallskip
\noindent
As first observed in \cite{be}, by choosing the set of atoms $\XX$ with an appropriate cardinality,
one can define a star map $*:\VV(\XX)\to\VV(\YY)$ where $\YY=\XX$, and hence
where the standard and the nonstandard universe coincide.
For completeness, let us review here the construction (see also \cite[A.1.4]{dgl}).

\begin{theorem}\label{ultrapowermodel}
Let $\XX$ be a set of atoms where $\R\subset\XX$, and assume
that $|\XX|=|\XX\setminus\R|=2^\kappa$ for some infinite cardinal $\kappa$.
Then there exists a model of nonstandard analysis
$\langle\VV(\XX),\VV(\XX),*\rangle$ where the standard and nonstandard universe coincide,
and where the $\kappa^+$-saturation property holds.
\end{theorem}

\begin{proof}
Consider the ``bounded" ultrapower $\VV(\XX)_b^{\kappa}/\U:=\bigcup_{n=0}^\infty{V_n(\XX)}^{\kappa}/\,\U$ of
the standard universe modulo a good countably incomplete
ultrafilter $\U$ on the cardinal $\kappa$.\footnote
{~Good ultrafilters $\U$ on an infinite cardinal $\kappa$
are a special class of countably incomplete ultrafilters
that can be characterized in model-theoretic terms by the following property:
``Every ultrapower modulo $\U$ of a first-order structure in a language $|\mathcal{L}|\le\kappa$
is $\kappa^+$-saturated." See \cite[\S 6.1]{ck}.}
Now the ultrafilter $\U$ is regular, so
any $\U$-ultrapower of an infinite set has the greatest possible cardinality, and therefore
$|\XX^\kappa/\U|=|\XX|^\kappa=(2^\kappa)^\kappa=2^\kappa=|\XX|$.\footnote
{~Recall that every good countably incomplete ultrafilter is regular
(see \cite[Prop.\,10.2]{keisler2}).}
Now let $d:\VV(\XX)\to\VV(\XX)_b^{\kappa}/\U$ be the diagonal embedding.
By our assumptions we can pick a bijection $\tau_0:\XX^\kappa/\U\to\XX$ such that
$\tau_0(d(x))=x$ for every $x\in\R$.
Then define the \emph{Mostowski's transitive collapse}
$\tau:\VV(\XX)_b^{\kappa}/\U\to T\subset\VV(\XX)$ by induction as follows:
\begin{itemize}
\item
If $f:\kappa\to\XX$, let $\tau([f]_\U)=\tau_0([f]_\U)$;
\item
If $f:\kappa\to V_n(\XX)$, let $\tau([f]_\U)=\{\tau([g]_\U)\mid\{i\in\kappa\mid g(i)\in f(i)\}\in\U\}$.
\end{itemize}

We observe that if
$\text{range}(f)\subseteq V_{n+1}(\XX)$ and
$\{i\in\kappa\mid g(i)\in f(i)\}\in\U$ then $[g]_\U\in V_n(\XX)^\kappa/\U$,
so the inductive definition is well-posed.
It directly follows from the definition that $T:=\text{range}(\tau)$ is a transitive
subset of $\VV(\XX)$. Finally, let $*=\tau\circ d$.
\[
\begin{tikzcd}[row sep=2.5em]
& \VV(\XX)^\kappa/\,\U \arrow{dr}{\tau} \\
\VV(\XX) \arrow{ur}{d} \arrow{rr}{*} && \ T\subseteq\VV(\XX)
\end{tikzcd}
\]

Note that $\tau(d(\XX))=\{\tau_0([g]_\U)\mid\{i\in\kappa\mid g(i)\in\XX\}\in\U\}=\tau_0(\XX^\kappa/\U)=\XX$,
and hence ${}^*\XX=\XX$.

It is easily verified that
$T=\bigcup_{n\ge 0}{}^*V_n(\XX)=\bigcup_{X\in\VV(\XX)}\hX={}^*\VV(\XX)$.
Now for every $n$, the diagonal embedding $d_n:V_n(\XX)\to V_n(\XX)^\kappa/\U$
is an elementary embedding, and so it follows that
$d=\bigcup_{n\ge 0}d_n:\VV(\XX)\to\VV(\XX)_b^\kappa/\U$ is a bounded
elementary embedding. Since $\tau$ is an isomorphism,
we can conclude that the map $*:(\VV(\XX),\in)\to(T,\in)$ is an elementary embedding.

Now recall the basic fact that bounded quantifier formulas are absolute
between transitive models (see, \emph{e.g.}, \cite[Ch. 12]{je}).
This means that if $T\subset T'$ are transitive,
then for every bounded quantifier formula $\varphi(x_1,\ldots,x_n)$
and for all $t_1,\ldots,t_n\in T$, one has the equivalence
$$T\models\varphi(t_1,\ldots,t_n)\Leftrightarrow T'\models\varphi(t_1,\ldots,t_n).$$

So, the \emph{transfer principle} holds for
the inclusion $\imath:T\to\VV(\XX)$,
and hence also for the map $*:\VV(\XX)\to\VV(\XX)$ seen as
taking values in the whole superstructure $\VV(\XX)$.
Note that, by the definition, we have
${}^*x=\tau(d(x))=x$ for every $x\in\R$.

Finally, the hypothesis that $\U$ is a countably incomplete good ultrafilter
on $\kappa$ guarantees that for every $n$
the ultrapower $V_n(\XX)^\kappa/\U$ is a $\kappa^+$-saturated $\in$-structure
in model-theoretic sense. If $\G\subseteq\hA$ is a bounded
family of internal sets, $|\G|\leq\kappa$,
and $n$ is such that
$A\subseteq V_n(\XX)$, then clearly $\G\subset\tau[V_n(\XX)^\kappa/\U]$.
Since $\G$ has the finite intersection property, the model-theoretic type
$$p(x):=\left\{x\in\bigcap\G_0\mid\G_0\subseteq\G\,\wedge\,|\G_0|<\omega\right\}$$
is consistent.
By $\kappa^+$-saturation of $V_n(\XX)^\kappa/\U$ and $|p(x)|\leq\kappa$, it follows
that $p(x)$ is satisfiable in $V_n(\XX)^\kappa/\U$.
Hence, $\bigcap_{B\in\G}B\ne\emptyset$, as desired.
\end{proof}

In the sequel we will work in a model of nonstandard analysis as given by
the previous theorem. So, we will work with a star map
$$*:\VV(\XX)\to\VV(\XX)$$
from a superstructure into itself, so that
there will be no distinction between the ``standard" and the "nonstandard" universe.
This framework has several advantages.

To begin with, since $\text{range}(*)\subseteq\VV(\XX)$,
the star map can be applied repeatedly without extra effort.
Besides, also external sets that are considered in the practice
belong to the superstructure, and so one can apply the star map to them.
For instance, if $\mu$ is a counting measure on a
hyperfinite subset of $\hR\subset\XX$,
then it is easily verified that the corresponding
Loeb measure space $\mathcal{L}$ belongs to the superstructure $\VV(\XX)$,
and so one can consider its nonstandard extension ${}^*\!\mathcal{L}$.

Note that when $*$ is applied
the second time, it is still a bounded elementary
embedding from $\VV(\XX)$ to $^*\VV(\XX)$; on the other hand,
the restriction $*\res ^*\VV(\XX)$ is now a bounded elementary
embedding from the internal universe $^*\VV(\XX)$
into the ``second" internal universe
$$^{**}\VV(\XX)=\bigcup_{X\in\VV(\XX)}{}^{**}\!X.$$

As a consequence, we obtain another interesting feature of our framework, namely
the fact that \emph{transfer} also applies to formulas that contain external relations as predicates.
Indeed, if $\A\subseteq\VV(\XX)\setminus\,^*\VV(\XX)$
is any family of external sets made of internal elements
(that is, if $a\subseteq\,^*\VV(\XX)$ for every $a\in\A$),
then one can view the second application of the star map as a bounded elementary embedding
from the structure $(^*\VV(\XX);A)_{A\in\A}$ to the structure $(^{**}\VV(\XX);\,\hA)_{A\in\A}$.\footnote
{~Here $(^*\VV(\XX);A)_{A\in\A}$ is the expanded model obtained by adding a unary relation $A$
to the model $^*\VV(\XX)$ for every $A\in\A$; and similarly for $(^{**}\VV(\XX);\,\hA)_{A\in\A}$.}
For example, the restriction
$$*\res{}^*\VV(\XX):({}^*\VV(\XX);\R)\to(^{**}\VV(\XX);\hR)$$
is a bounded elementary embedding,
although $\R$, as a unary relation in $^*\VV(\XX)$,
is not an element of $^*\VV(\XX)$.
We remark that fundamental ``external" objects in nonstandard analysis,
such as the standard part map or the Loeb measure on a hyperfinite set,
are definable in the structure $({}^*\VV(\XX);\R)$.

As another example, consider the
bounded elementary embedding given by
the restriction of the star map seen as taking values in the internal universe:
$$*:\VV(\XX)\to{}^*\VV(\XX).$$
For every $X\in\VV(\XX)$, the restriction $*\res X$, as a subset of $X\times\hX$,
is an element of $\VV(\XX)$, and so the star-map $*$ can be applied to it
so as to obtain a map $^*(*\res X):\hX\to{}^{**}\!X$.
This will be a key property in our definitions of the internal star maps
given in \S\ref{sec-bistarmap} and \S\ref{sec-internalstarmaps}.

\smallskip
Finally, observe that in our framework where the star map
is defined from a universe into itself,
the saturation property transfers to an internal version of saturation.

\begin{proposition}
Assume that $\kappa$-saturation holds for an infinite cardinal $\kappa$.
If $\kappa\subseteq\XX$ then the following
``internal ${}^*\kappa$-saturation" holds:
\begin{itemize}
\item
Let $f:{}^*\kappa\to{}^*\mathcal{P}(A)$ be
an internal ${}^*\kappa$-sequence of internal subsets of a set $\hA\in{}^*\VV(\XX)$,
\emph{i.e.}\,$f\in{}^*\text{Fun}(\kappa,\mathcal{P}(A))$.
If for every ${}^*$finite $F\subset{}^*\kappa$ the intersection
$\bigcap_{\alpha\in F}f(\alpha)\ne\emptyset$ then the
whole intersection $\bigcap_{\alpha\in{}^*\kappa}f(\alpha)\ne\emptyset$.\footnote
{~Recall that $F$ is a ${}^*$finite subset of ${}^*\kappa$ means that
$F\in{}^*\text{Fin}(\kappa)$, \emph{i.e.} $F$
belongs to the star extension of the family of finite subsets of $\kappa$.}
\end{itemize}
\end{proposition}

\begin{proof}
It is a straightforward application of \emph{transfer}
to the the $\kappa$-saturation property.
\end{proof}

\section{The internal star map}\label{sec-bistarmap}

In this section we introduce the ``internal star map" $\circledast$,
and give foundations to its use, starting from the proof of the \emph{transfer principle}
for $\circledast$. In some precise sense, we can say that the internal star map is the star map
as seen by someone who lives within the internal universe:
$${}^*\VV(\XX)=\bigcup_{X\in\VV(\XX)}\hX.$$

Then we will generalize the notion of internal star map to $n$-th internal star map
for every $n\in\N$, and we will consider iterations and compositions of internal
star maps with the ``external'' star map.

\smallskip
Let us start by focusing on the \emph{double} hyper-extension ${}^{**}\!X={}^*({}^*\!X)$
of an arbitrary object $X\in\VV(\XX)$. We observe that there are
two different canonical ways to embed
the single hyper-extension $\hX$ into ${}^{**}X$.

\smallskip
Clearly, the first one is provided by the star map:
$$*:\xi\longmapsto{}^*\xi.$$

However, we will see next that there is another canonical
way of embedding $\hX$ into ${}^{**}\!X$.

\smallskip
\begin{itemize}
\item[$\blacklozenge$]
For convenience of notation, in the following we will use the symbol $\s$ for the star map $*$.
\end{itemize}

For any set $X\in\VV(\XX)$, consider the restriction of the star map to $X$:
$$\s\res X:X\to\hX$$

The crucial observation here is that
the function $\s\res X$ is itself an element of our universe $\VV(\XX)$,
and so it makes sense to consider its hyper-extension ${}^*(\s\res X):\hX\to{}^{**}X$.

We stress the fact that in general ${}^*(\s\res X)(\xi)\ne (\s\res\hX)(\xi)$
(see Proposition \ref{star=dast} below).

\begin{definition}\label{def-bistarmap}
The \emph{internal star map}
$\circledast:{}^*\VV(\XX)\to\VV(\XX)$ with domain the internal
universe ${}^*\VV(\XX)$, is defined by putting:
$$\circledast:=\bigcup_{X\in\VV(\XX)}{}^*(\s\res X).$$
\end{definition}

In other words, if $\xi\in\hX$ then ${}^\circledast\xi$
is the image of $\xi$ under the hyper-extension of the function $\s\res X$.
The element ${}^\circledast\xi$
is called the \emph{internal star-extension} of $\xi$.\footnote
{~As done with the star map, for simplicity we will
write ${}^\circledast x$ to mean $\circledast(x)$.}

\smallskip
The above definition is well-posed.
Indeed, if $\xi\in\hX$ and $\xi\in\hY$, then
it is readily seen by the definitions that both functions
$\s\res X$ and $\s\res Y$ extend the function $\s\res(X\cap Y)$
and so, by {\it transfer}, ${}^*(\s\res X)(\xi)={}^*(\s\res(X\cap Y))(\xi)={}^*(\s\res Y)(\xi)$.

\smallskip
We remark that the map $\circledast$ makes sense because we
assumed that our star map is defined from a universe into itself.

\smallskip
The composition of the star map $*$ with the internal star map $\circledast$ equals
the double star map $**$.

\begin{proposition}\label{double}
For every $x\in\VV(\XX)$ one has that ${}^{\circledast *}x={}^{**}x$.
\end{proposition}

\begin{proof}
By definition, ${}^\circledast({}^*x)={}^*(\s\res X)({}^*x)$, and the
latter expression equals ${}^*((\s\res X)(x))={}^*({}^*x)$.
\end{proof}

The star map and the internal star map only coincide on hyper-extensions.

\begin{proposition}\label{star=dast}
Let $\xi,\eta\in\hX$. Then ${}^\circledast\xi={}^*\eta$ if and only if
there exists $x\in X$ such that $\xi=\eta={}^*x$.
\end{proposition}

\begin{proof}
If $\xi=\eta={}^*x$ for some $x\in X$ then, by using the previous proposition,
we have ${}^\circledast\xi={}^\circledast({}^*x)={}^*({}^*x)={}^*\eta$.

Conversely, note that ${}^*\eta={}^\circledast\xi={}^*(\s\res X)(\xi)\Leftrightarrow
\xi\in{}^*\Lambda$ where
$$\Lambda:=\{x\in X\mid (\s\res X)(x)=\eta\}.$$
Since ${}^*\Lambda\ne\emptyset$, we can pick
an element $x\in\Lambda\ne\emptyset$ such that $(\s\res X)(x)={}^*x=\eta$.
The map $\s\res X$ is 1-1, and so such an $x$ is unique and $\Lambda=\{x\}$.
Finally, $\xi\in{}^*\Lambda=\{{}^*x\}$ implies that $\xi={}^*x$ and
we can conclude that $\xi=\eta={}^*x$.
%
%
%
\end{proof}

%

\begin{remark}\label{pullback}
Propositions \ref{star=dast} states that the following
commutative diagram is the \emph{pullback} in the category of sets
of the two embeddings $*,\circledast:\hX\to{}^{**}\!X$:
$$\begin{tikzcd}
X \arrow{r}{*}
\arrow{d}{*} &
{}^*\!X \arrow{d}{\circledast} \\
{}^*\!X \arrow{r}{*} & {}^{**}\!X
\end{tikzcd}$$
\end{remark}

\smallskip
The internal star map $\circledast$ satisfies the transfer principle.

\begin{theorem}[Internal Transfer Principle]\label{dualtransfer}
Let $\sigma(x_1,\ldots,x_k)$ be a bounded quantifier formula
and let $\xi_1,\ldots,\xi_k\in{}^*\VV(\XX) $ be internal elements. Then
$$\sigma(\xi_1,\ldots,\xi_k)\ \Longleftrightarrow\
\sigma({}^\circledast\xi_1,\ldots,{}^\circledast\xi_k).$$
\end{theorem}

\begin{proof}
Pick $X$ such that $\xi_1,\ldots,\xi_k\in\hX$.
By the definition of internal star map,
$\sigma({}^\circledast\xi_1,\ldots,{}^\circledast\xi_k)$ is the same
as $\sigma({}^*(\s\res X)(\xi_1),\ldots,{}^*(\s\res X)(\xi_k))$ and so
$$\sigma({}^\circledast\xi_1,\ldots,{}^\circledast\xi_k)\ \Longleftrightarrow\
(\xi_1,\ldots,\xi_k)\in{}^*\Gamma$$
where $\Gamma:=\{(x_1,\ldots,x_k)\in X^k\mid
\sigma((\s\res X)(x_1),\ldots,(\s\res X)(x_k))\}$.
Now, for all $x_1,\ldots,x_k\in X$, the formula
$\sigma((\s\res X)(x_1),\ldots,(\s\res X)(x_k))$ is the same as
$\sigma({}^*x_1,\ldots,{}^*x_k)$ which in turn, by \emph{transfer},
is equivalent to $\sigma(x_1,\ldots,x_k)$, and so
$\Gamma=\{(x_1,\ldots,x_k)\in X^k\mid \sigma(x_1,\ldots,x_k)\}$.
Finally, notice that $(\xi_1,\ldots,\xi_k)\in{}^*\Gamma\Leftrightarrow
\sigma(\xi_1,\ldots,\xi_k)$.
\end{proof}

\subsection{The internal star map on the hyperreal numbers}

\

\noindent
\smallskip
Since we assumed that ${}^*r=r$ for all real numbers $r$,
the internal star map $\circledast:\hR\to{}^{**}\R$
on the hyperreals is just the inclusion.

\begin{proposition}\label{circledastidentityonhR}
${}^\circledast\xi=\xi$ for every $\xi\in\hR$.
\end{proposition}

\begin{proof}
By \emph{transfer} from the bounded quantifier formula ``$\forall x\in\R\ (\s\res\R)(x)=x$''
we obtain that ``$\forall \xi\in\hR\ {}^\circledast\xi={}^*(\s\res\R)(\xi)=\xi$''.
\end{proof}

Note that the above result is consistent with the intuition that $\circledast$
is the star map as seen ``internally''.
Indeed, the hyperreal numbers $\hR$ are the real numbers
as seen by someone who lives within the internal universe;
since the star map does not move real numbers,
it follows that the internal star map $\circledast$ does not move
elements of $\hR$.

\smallskip
On the hyperreal line $\hR$, one considers the order relation
given by the hyper-extension ${}^*\!\!\!<$ of the
order $<$ on $\R$. Since ${}^*r=r$ for every $r\in\R$,
the relation ${}^*\!\!\!<$ is an extension of $<$ and so,
to simplify matters, one usually omits the star symbol
and directly writes $<$ also for the order on $\hR$.
In the double hyper-extension ${}^{**}\R$,
the order relations ${}^{\circledast *}\!\!\!<$ and ${}^{**}\!\!\!<$
coincide, and so also in this case one can omit the star and internal star symbols and
directly write $<$ with no ambiguity.

\begin{proposition}\label{circledastreal}
\

\begin{enumerate}
\item
If $\xi,\zeta\in\hR_+$ are positive infinite,
then ${}^*\xi>\zeta$.
\item
If $\varepsilon,\vartheta\in\hR_+$ are positive infinitesimals,
then ${}^*\varepsilon<\vartheta$.
\end{enumerate}
\end{proposition}

\begin{proof}
(1). By \emph{transfer} from ``$\forall x\in\R_+\ (\xi>x)$" one obtains
that ``$\forall x\in\hR_+\ ({}^*\xi>x)$" and so, in particular,
${}^*\xi>\zeta$.

(2). Proceed as above, by applying \emph{transfer} to the statement
``$\forall x\in\R_+\ (\varepsilon<x)$".
\end{proof}

In the case of natural numbers, we have the following properties,
that will be relevant in the sequel:

\begin{proposition}
\

\begin{enumerate}
\item
${}^\circledast\nu=\nu$ for every $\nu\in\hN$.
\item
${}^*\nu>\nu$ for every $\nu\in\hN\setminus\N$.
\item
$\hN$ is an initial segment of ${}^{**}\N$, \emph{i.e.}
$\nu<\Omega$ for every $\nu\in\hN$ and every $\Omega\in{}^{**}\N\setminus\hN$.
\end{enumerate}
\end{proposition}

\begin{proof}
$(1)$ is a particular case of Proposition \ref{circledastidentityonhR};
and $(2)$ is a particular case of Proposition \ref{circledastreal} (1).
Finally, $(3)$ is obtained by \emph{transfer} from the property:
``$\forall x\in\N\ \forall y\in\hN\setminus\N\ (x<y)$".
\end{proof}

\smallskip
Although
${}^\circledast\xi=\xi={}^*\xi$ for every $\xi\in\R$,
we remark that the internal star map $\circledast$ and the star map $*$
are in general different when applied to internal subsets
$B\subseteq\hR$.
Indeed, recall that by Propositions \ref{double} and \ref{star=dast},
one has that ${}^\circledast B={}^*B$ if and only if $B=\hA$ for some $A\subseteq\R$.

\begin{proposition}
If $A\subseteq\hR$ is an internal subset then $A\subseteq{}^\circledast A$.
\end{proposition}

\begin{proof}
Apply \emph{transfer} to the property:
``$\forall\,A\in\PP(\R)\  \hA=(\s\res{\PP(\R)})(A)\supseteq A$", and obtain
``$\forall\,A\in{}^*\PP(\R)\  {}^\circledast A={}^*(\s\res{\PP(\R)})(A)\supseteq A$".
\end{proof}

Recall that we are assuming ${}^*r=r$ for every $r\in\R$, and hence the
inclusion $A\subseteq\hA$ holds for every $A\subseteq\R$.
However, for subsets $A\subseteq\hR$ of the hyperreal numbers,
we remark the same property does not hold; \emph{e.g.},
if $A=[\xi,2\xi]\subset\hR$ where $\xi>0$ is infinite (or infinitesimal),
then $A$ is an interval set and
$A\cap\hA=\emptyset$ because every element of $\hA=[{}^*\xi,2\,{}^*\xi]$
is larger (smaller, respectively) than any element of $A$.
(See the Example below.)

Note that, by \emph{transfer},
one has the inclusion $\{{}^*a\mid a\in A\}\subseteq\hA$ for every $A\subseteq\hR$.

\begin{example}
{\rm Let $\xi<\zeta$ be positive infinite hyperreal numbers in $\hR$,
and consider the star-extension of the interval $[\xi,\zeta]$:
$${}^*[\xi,\zeta]=\{x\in{}^{**}\R\mid {}^*\xi<x<{}^*\zeta\}$$
and the internal star-extension of the same interval $[\xi,\zeta]$:
$${}^\circledast[\xi,\zeta]=\{x\in{}^{\circledast *}\R\mid
{}^\circledast\xi<x<{}^\circledast\zeta\}=
\{x\in{}^{**}\R\mid \xi<x<\zeta\}.$$
Then ${}^*[\xi,\zeta]$ and ${}^\circledast[\xi,\zeta]$
are disjoint, since every element of the former is greater than any element of the latter.
Similarly, if $\varepsilon<\eta$ are positive infinitesimal hyperreal numbers
in $\hR$, then every element of the interval
${}^*[\varepsilon,\eta]=\{x\in{}^{**}\R\mid {}^*\varepsilon<x<{}^*\eta\}$
is smaller than any element of
${}^\circledast[\varepsilon,\eta]=\{x\in{}^{**}\R\mid\varepsilon<x<\eta\}$.}
\end{example}

\subsection{The internal star map in the ultrapower model}

\

\noindent
%
%
%
Let $\U$ be a non-principal ultrafilter on a set $I$,
and let $*$ be the star-map on $\VV(\XX)$
as induced by the diagonal embedding
$$d:\VV(\XX)\to\VV(\XX)_b^I/\U\cong{}^*\VV(\XX)$$
of the universe into its \emph{bounded} ultrapower
$\VV(\XX)_b^I/\U:=\bigcup_{n\ge 0}V_n(\XX)^I/\U$,
according to the construction seen in the proof of Theorem \ref{ultrapowermodel}.
Then the star-map $*$ on ${}^*\VV(\XX)$
can be seen as induced by the diagonal embedding into the double ultrapower
$$\dd:\VV(\XX)_b^I/\U\to(\VV(\XX)_b^I/\U)_b^I/\U\cong{}^{**}\VV(\XX).$$
Precisely, given any bounded function $f:I\to\VV(\XX)$,
the diagonal embedding $\dd$ maps the $\U$-equivalence class
$[f]_\U\in\VV(\XX)_b^I/\U$ to the $\U$-equivalence class
of the constant sequence with value $[f]_\U$, that is:
$$\dd:[f]_\U\longmapsto[([f]_\U\mid i\in I)]_\U\in(\VV(\XX)_b^I/\U)_b^I/\U.$$
The double nonstandard embedding $**:\VV(\XX)\to{}^{**}\VV(\XX)$ can be seen as induced
by the composition $\dd\circ d:\VV(\XX)\to(\VV(\XX)_b^I/\U)_b^I/\U$.

Besides $\dd$, there is another canonical embedding
of the ultrapower \newline $\VV(\XX)_b^I/\U$
into the double ultrapower $(\VV(\XX)_b^I/\U)_b^I/\U$
that corresponds to the internal star-map $\circledast$,
namely the function
$$\widehat{d}:\VV(\XX)_b^I/\U\to(\VV(\XX)_b^I/\U)_b^I/\U\cong{}^{**}\VV(\XX)$$
obtained by ``lifting" the diagonal embedding $d$.
Precisely, for every bounded function $f:I\to \VV(\XX)$,
the embedding $\widehat{d}$ maps the $\U$-equivalence class $[f]_\U\in\VV(\XX)_b^I/\U$
to the $\U$-equivalence class of the sequence of the diagonal embeddings in $\VV(\XX)_b^I/\U$
of the elements $f(i)$, that is:
$$\widehat{d}:[f]_\U\longmapsto [(d(f(i))\mid i\in I)]_\U=[d\circ f]_\U\in(\VV(\XX)_b^I/\U)_b^I/\U $$
It can be directly verified from the definitions that the following
diagram gives a pullback for the embeddings
$\dd$ and $\widehat{d}$ in the category of sets:\footnote
{~Actually, it is a pullback in the category of models of the
language of set theory, with bounded elementary embeddings as arrows.}
$$\begin{tikzcd}
\VV(\XX) \arrow{r}{d}
\arrow{d}{d} &
\VV(\XX)_b^I/\U \arrow{d}{\widehat{d}} \\
\VV(\XX)_b^I/\U \arrow{r}{\dd} & (\VV(\XX)_b^I/\U)_b^I/\U
\end{tikzcd}$$
Clearly, the above diagram correspond to the pullback:
$$\begin{tikzcd}
\VV(\XX) \arrow{r}{*}
\arrow{d}{*} &
{}^*\VV(\XX) \arrow{d}{\circledast} \\
{}^*\VV(\XX) \arrow{r}{*} & {}^{**}\VV(\XX)
\end{tikzcd}$$

\begin{remark}
Assume the nonstandard model has been constructed by
using a bounded ultrapower modulo the ultrafilter $\U$ on the set $I$,
and assume that $\U,I\in\VV(\XX)$.
For every $n$, the internal star map $\circledast$ restricted to ${}^*V_n(\XX)$
can be seen as induced by the diagonal embedding from
${}^*V_n(\XX)$ into its internal ultrapower
modulo the internal ultrafilter $^*\U$ on the internal set ${}^*I$.
\end{remark}

\begin{remark}
Up to isomorphisms, the composition $**=\dd\circ d=\widehat{d}\circ d=\circledast *$ can be
seen as the diagonal embedding
$$d':\VV(\XX)\to\VV(\XX)_b^{I\times I}/\U\otimes\U.$$
Indeed, there is a canonical isomorphism
between $\VV(\XX)_b^{I\times I}/\U\otimes\U$ and the double
ultrapower $(\VV(\XX)_b^I/\U)_b^I/\U\cong{}^{**}\VV(\XX)$.\footnote
{~See \cite{ck, jin}. Recall that
$\U\otimes\U=\{A\subseteq I\times I\mid \{i\in I\mid \{j\in I\mid (i,j)\in A\}\in\U\}\in\U\}$
is the \emph{tensor product} of $\U$ with itself.}
\end{remark}

\section{Internal star maps}\label{sec-internalstarmaps}

In this section we generalize the definition of internal star map $\circledast$.
To simplify matters, in the following:

\begin{itemize}
\item
We will simply write $\VV$ instead of $\VV(\XX)$,
\item
We will write $\s$ for the star map $*$.
\item
We will write $\s^n$ to denote the $n$-th iteration of the star map;
so, $\s(A)=\hA$, $\s^2(A)={}^{**}A$, and so forth.
\end{itemize}

\begin{definition}
For every $n\in\N$, the $n$-th \emph{internal universe} $\VV_n$ is the union
$$\VV_n:=\bigcup_{X\in\VV}\s^n(X).$$
\end{definition}

Recall that the universe $\VV=\bigcup_{k\ge 0}V_k(\XX)$ was
defined as the union of the finite ``levels" $V_k(\XX)$.
Note that the subscript $n$ in $\VV_n$
does not mean a ``level'' of one universe $\VV$. Instead, it means
that $\VV_n$ is the $n$-th universe in the sequence of universes
$\VV,\VV_1,\VV_2,\ldots$.

\begin{proposition}
$\s^n(\N),\s^n(\R)\in\VV_n$
but $\s^n(\N),\s^n(\R)\notin\VV_{n+1}$.
\end{proposition}

\begin{proof}
It is well known in nonstandard analysis that $\N$ and $\R$
are external objects, \emph{i.e.}
$\N,\R\notin\hX$ for any $X\in\VV$.
Then, by applying \emph{transfer} to the iterated star map $\s^n$,
we obtain that $\s^n(\N),\s^n(\R)\notin\s^{n+1}(X)$ for any $X$,
and hence $\s^n(\N),\s^n(\R)\notin\VV_{n+1}$.
On the other hand $\N,\R\in\VV$ and so $\s^n(\N),\s^n(\R)\in\VV_n$.
\end{proof}

\begin{proposition}\label{structure1}
Let $n\in\N$. Then:

\begin{enumerate}
\item
$\s^n(\XX)=\XX$ for every $n$.
\item
$\VV_n=\bigcup_{k\ge 0}\s^n(V_k(\XX))$.
\item
$\VV_{n+1}\subset\VV_n$, and hence the internal universes form a (strictly) decreasing chain:
$\VV\supset\VV_1\supset\VV_2\supset\ \ldots\ \supset\VV_n\supset\VV_{n+1}\supset\ \ldots$.
\end{enumerate}
\end{proposition}

\begin{proof}
(1).\ This is a trivial fact based on the assumption that $\s(\XX)=\XX$.

\smallskip
(2).\ One inclusion is trivial because every level
$V_k(\XX)\in\VV$.
For the converse inclusion $\bigcup_{X\in\VV}\s^n(X)\subseteq\bigcup_{k\ge 0}\s^n(V_k(\XX))$,
observe that for every $X\in\VV$ we can pick $k$ such that $X\subseteq V_k(\XX)$,
and hence $\s^n(X)\subseteq\s^n(V_k(\XX))$.


\smallskip
(3).\ For every $X\in\VV$, we have that $\hX\in\VV$, and hence
$\hX\subseteq V_k(\XX)$ for some $k$. Then
by \emph{transfer} with respect to $\s^n$, we
obtain that $\s^{n+1}(X)\subseteq\s^n(V_k(\XX))\subseteq\VV_n$.
This shows that $\VV_{n+1}\subseteq\VV_n$.
The inclusion is proper because, \emph{e.g.},
$\s^n(\N)\in\VV_n$ but $\s^n(\N)\notin\VV_{n+1}$, as shown in the previous proposition.
\end{proof}

\begin{definition}
For every $n\in\N$, the $n$-th \emph{internal star map} is the
function $\i_n:\VV_n\to\VV_{n+1}$ defined by putting:
$$\i_n:=\s^n(\s)=\bigcup_{X\in\VV}\s^n(\s\res X).$$
Elements of $\VV_n$ will be called \emph{$n$-th internal} elements.
 \end{definition}

\begin{remark}\label{intuitioninternalstarmap}
Intuitively, one can think of $\i_n$ as the star map when $\VV_n$ is
taken as the ``standard'' universe; in other words, $\i_n$ is the star map as viewed
by someone who lives within the $n$-th internal universe $\VV_n$.
\end{remark}

Let us check that the above definition is well-posed and that $\i_n$ takes values in $\VV_{n+1}$.
For every $X\in\VV$, the $n$-th internal star-map extends
the following function:
$$\s^n(\s\res X):\s^n(X)\to\s^{n+1}(X).$$
Since for all $X,Y\in\VV$,
one has $\s\res(X\cap Y)=(\s\res X)\cap(\s\res Y)$,
it follows that $\s^n(\s\res(X\cap Y))=\s^n(\s\res X)\cap\s^n(\s\res Y)$.
Besides, it is clear that $\text{range}(\i_n)\subseteq\bigcup_{X\in\VV}\s^{n+1}(X)=\VV_{n+1}$.

\smallskip
Note that the \emph{first} internal universe
$\VV_1=\bigcup_{X\in\VV}\hX$ and the \emph{first}
internal map $\i_1$ are the internal universe ${}^*\VV(\XX)$
and the internal star map $\circledast$
considered in \S\ref{sec-bistarmap}, respectively.

\begin{proposition}\label{rangeinternalmaps}
For every $\ell\in\N$, one has that
$\i_n[\VV_{n+\ell}]\subseteq\VV_{n+\ell+1}$.
\end{proposition}

\begin{proof}
Let $a\in\VV_{n+\ell}$. Pick $X\in\VV$ such that $a\in\s^{n+\ell}(X)$,
and let $Y=\s^\ell(X)$.
Observe that the restriction of the star map
$\s\res Y:\s^\ell(X)\to\s^{\ell+1}(X)$, and so
$\i_n(a)=\s^n(\s\res Y)(a)\in\s^n(\s^{\ell+1}(X))=\s^{n+\ell+1}(X)\subseteq\VV_{n+\ell+1}$.
\end{proof}

\smallskip
Similarly to the internal star map $\circledast$, every $n$-th internal star map $\i_n$ satisfies
the transfer principle.

\begin{theorem}[$n$-th Internal Transfer Principle]\label{n-dualtransfer}
Let $\sigma(x_1,\ldots,x_k)$ be a bounded quantifier formula
and let $\xi_1,\ldots,\xi_k\in\VV_n$ be $n$-th internal elements. Then
$$\sigma(\xi_1,\ldots,\xi_k)\ \Longleftrightarrow\
\sigma(\i_n(\xi_1),\ldots,\i_n(\xi_k)).$$
\end{theorem}

\begin{proof}
The proof is just a straight modification of that of Theorem \ref{dualtransfer},
where one considers the $n$-iterated star-extensions $\s^n(\s\res X)$ in place of
the (single) star-extensions ${}^*(\s\res X)=\s\,(\s\res X)$.

Precisely, pick $X$ such that $\xi_1,\ldots,\xi_k\in\s^n(X)$.
By the definition of $n$-th internal star map,
$\sigma(\i_n(\xi_1),\ldots,\i_n(\xi_k))$ is
$\sigma(\s^n(\s\res X)(\xi_1),\ldots,\s^n(\s\res X)(\xi_k))$, and so
$$\sigma(\i_n(\xi_1),\ldots,\i_n(\xi_k))\ \Longleftrightarrow\
(\xi_1,\ldots,\xi_k)\in\s^n(\Gamma)$$
where $\Gamma:=\{(x_1,\ldots,x_k)\in X^k\mid
\sigma((\s\res X)(x_1),\ldots,(\s\res X)(x_k))\}$, that is
$$\Gamma=\{(x_1,\ldots,x_k)\in X^k\mid \sigma({}^*x_1,\ldots,{}^*x_k)\}.$$
Note that for all $x_1,\ldots,x_k\in X$, by \emph{transfer}, $\sigma({}^*x_1,\ldots,{}^*x_k)$ is equivalent to
$\sigma(x_1,\ldots,x_k)$, and so
$\Gamma=\{(x_1,\ldots,x_k)\in X^n\mid\sigma(x_1,\ldots,x_k)\}$. We finally conclude that
$\s^n(\Gamma)=\{(\xi_1,\ldots,\xi_k)\in\s^n(X)\mid\sigma(\xi_1,\ldots,\xi_k)\}$.
\end{proof}

\subsection{Compositions of internal star maps}
\

\noindent
In this subsection, we will consider compositions of
$n$-th internal star maps with the (external) star map,
and prove a few equalities about them.

Similarly as done with the star map $\s$, we adopt the following notation.

\begin{itemize}
\item
We will write $(\i_n)^m$ to denote the $m$-th iteration of the $n$-th internal star map;
so, $(\i_n)^1(A)=\i_n(A)$, and inductively $(\i_n)^{m+1}(A)=\i_n((\i_n)^m(A))$.
\end{itemize}

We will use the following properties, that
are straight consequences of the \emph{transfer principle}.

\begin{proposition}\label{transferaboutfunctions}
Let $n\in \N$.

\begin{enumerate}
\item
For every function $f$ and for every $a\in\text{dom}(f)$,
one has $\s^n(f(a))=\s^n(f)(\s^n(a))$ and $\i^n(f(a))=\i^n(f)(\i^n(a))$.
\item
For all functions $f$ and $g$ such that the composition $f\circ g$ is defined,
one has $\s^n(f\circ g)=\s^n(f)\circ\s^n(g)$ and $\i^n(f\circ g)=\i^n(f)\circ\i^n(g)$.
\end{enumerate}
\end{proposition}

According to the Remark \ref{intuitioninternalstarmap},
the $m$-th iteration of the $n$-th internal
star map would be the $m$-th iteration of the star map as seen within the universe $\VV_n$.
This intuition is formalized by the following property:

\begin{proposition}\label{compositions1}
For every $n,m\in\N$, the $m$-th iteration of the $n$-th
internal star-map $(\i_n)^m:\VV_n\to\VV_{n+m}$ is given by:
$$(\i_n)^m=
\s^n(\s^m):=
\bigcup_{X\in\VV}\s^n(\s^m\res X).$$
\end{proposition}

\begin{proof}
Note first that, since $\text{range}(\i_n)\subseteq\VV_{n+1}$,
by Proposition \ref{rangeinternalmaps} we have that
$\text{range}(\i_n\circ\i_n)\subseteq\i_n[\VV_{n+1}]\subseteq\VV_{n+2}$.
By iterating the argument, it is seen that $\text{range}((\i_n)^m)\subseteq\VV_{n+m}$
for every $m$.

Let us now fix $a\in\s^n(X)\in\VV_n$.
By induction on $m$ we want to show that
$(\i_n)^m(a)=\s^n(\s^m\res X)(a)$.
The base case $m=1$
is just the definition of $n$-th internal star map.
At the inductive step $m+1$, observe that
$\s^n(\s^m\res X):\s^n(X)\to\s^{n+m}(X)$.
Then, if let $Y=\s^{m}(X)$, we have that
\begin{multline*}
(\i_n)^{m+1}(a)=(\i_n\circ(\i_n)^m)(a)=
[\s^n(\s\res Y)\circ \s^n(\s^m\res X)](a)=
\\
=\s^n((\s\res Y)\circ(\s^m\res X))(a)=\s^n(\s^{m+1}\res X)(a).
\end{multline*}

By adopting a more compact notation where restrictions are not specified,
the above argument can be rewritten as follows:
$$(\i_n)^{m+1}=\i_n\circ(\i_n)^m=[\s^n(\s)]\circ[\s^n(\s^m)]=
\s^n(\s\circ\s^m)=\s^n(\s^{m+1}).$$
\end{proof}

For instance, the double internal star-map $\circledast\circledast:\VV_1\to\VV_3$,
\emph{i.e.} the composition of $\circledast$ with itself,
is obtained as the hyper-extension of the double star map $**$; that is:
$$\circledast\circledast=(\i_1)^2=\s(\s^2)=\bigcup_{X\in\VV}\s(\s^2\res X)$$
where $\s^2\res X:\xi\longmapsto{}^{**}\xi$ for every $\xi\in X$.


\smallskip
As we have seen in Proposition \ref{double}, we have that
$\circledast*=**$, that is, $\i_1\circ\s=\s^2$.
More generally, concerning the compositions of iterated star maps $\s^m$
with the $n$-th internal star map $\i_n$, the following equalities apply:

\begin{proposition}\label{compositions2}
Let $n\le m$. Then for every $\ell$, the composition
$$(\i_n)^\ell\circ\s^m=\s^{m+\ell}.$$
\end{proposition}

\begin{proof}
Let us first consider the case when $\ell=1$.

Let $a\in X\in\VV$. If $m=n$ then
$$(\i_n\circ\s^n)(a)=\i_n(\s^n(a))=
\s^n(\s\res X)(\s^n(a))=\s^n((\s\res X)(a))=\s^{n+1}(a).$$

By adopting a more compact notation where restrictions are not specified,
the above argument can be rewritten as follows:
$$\i_n\circ\s^n=
[\s^n(\s)]\circ\s^n=\s^n\circ\s=\s^{n+1}.$$

We now proceed by induction on $\ell$.
The base case $\ell=1$ has just been proved.
As for the inductive step, we observe that:
$$(\i_n)^{\ell+1}\circ\s^m=\i_n\circ[(\i_n)^\ell\circ\s^m]=\i_n\circ\s^{m+\ell}=\s^{m+\ell+1},$$
where the last equality again follows by the inductive base case.
\end{proof}

\begin{remark}
The composition $\i_n\circ\s^m$ is
not defined when $n>m$. Indeed, for every $X\in\VV$,
one has that $\s^m(X)\in\text{range}(\s^m)$, but
one can easily show that if $X\notin\VV_1$
then $\s^m(X)\notin\VV_{m+1}\subseteq\VV_n=\text{dom}(\i_n)$.
However, that composition is defined if one restricts to elements
of $\VV_{n-m}$. Indeed if $a\in\s^{n-m}(X)\in\VV_{n-m}$
then $\s^m(a)\in\s^n(X)\subseteq\VV_n$, and so the image $\i_n(\s^m(a))$ is defined.
We will call $\VV_{n-m}$ the natural domain of the map $\i_n\circ\s^m$ when
$n>m$ and assume implicitly that the domain of $\i_n\circ\s^m$ is $\VV_{n-m}$.
By the same reason we call $\VV_{k+n-m}$ the natural domain of the map
$\i_{k+n}\circ\i_k^m$ when $n>m$.
\end{remark}

\begin{proposition}\label{compositions2bis}
Let $n>m$. Then for every $\ell$, the restricted compositions to its natural domain
$$((\i_n)^\ell\circ\s^m)\res\VV_{n-m}=(\s^m\circ(\i_{n-m})^\ell)\res\VV_{n-m}.$$
\end{proposition}

\begin{proof}
%
Let us first consider the case when $\ell=1$.

Let $a\in\s^{n-m}(X)\in\VV_{n-m}$. We observe that
$\s^n(\s\res X):\s^{n}(X)\to\s^{n+1}(X)$
and $\s^{n-m}(\s\res X):\s^{n-m}(X)\to\s^{n-m+1}(X)$.
Since $\s^m(a)\in\s^{n}(X)$, then we have that:
\begin{multline*}
\i_n(\s^m(a))=\s^n(\s\res X)(\s^m(a))=\s^m(\s^{n-m}(\s\res X))(\s^m(a))=
\\
=\s^m(\s^{n-m}(\s\res X)(a))=\s^m(\i_{n-m}(a)).
\end{multline*}

In compact notation:
$$\i_n\circ\s^m=\s^n(\s)\circ\s^m=\s^m(\s^{n-m}(\s))\circ\s^m=
\s^m\circ\s^{n-m}(\s)=\s^m\circ\i_{n-m}.$$

We now proceed by induction on $\ell$.
The base case $\ell=1$ has just been proved.
As the inductive step, we observe that:
\begin{multline*}
(\i_n)^{\ell+1}\circ\s^m=\i_n\circ[(\i_n)^\ell\circ\s^m]=
\i_n\circ[\s^m\circ(\i_{n-m})^\ell]=
\\
=(\i_n\circ\s^m)\circ(\i_{n-m})^\ell=(\s^m\circ\i_{n-m})\circ(\i_{n-m})^\ell=
\s^m\circ(\i_{n-m})^{\ell+1}.
\end{multline*}

\end{proof}

%

%
%
%

\smallskip
According to Remark \ref{intuitioninternalstarmap},
living within the universe $\VV_k$,
the map $\i_k$ is viewed as the star map $*$, and thus the map $\i_{k+1}$
is viewed as the map $\i_1=\circledast$.
Consequently,
the equality $\circledast*=**$ corresponds to the
equality $\i_{k+1}\circ\i_k=\i_k\circ\i_k=(\i_k)^2$.
More generally, within $\VV_k$, the map $\i_{k+n}$
is viewed as the map $\i_n$, and thus the equality
of the previous proposition
corresponds to the following one.

\begin{proposition}\label{compositions3}
Let $n\le m$. Then
$(\i_{k+n})^\ell\circ(\i_k)^m=(\i_k)^{m+\ell}$ for all $k,\ell$.

\end{proposition}

\begin{proof}
We proceed by induction on $\ell$.
Towards the base case $\ell=1$, recall from Proposition \ref{compositions1} that
$(\i_k)^m=\s^k(\s^m)$ for every $m$.
Then, by Proposition \ref{compositions2}, we obtain:
\begin{multline*}
\i_{k+n}\circ(\i_k)^m=\s^{k+n}(\s)\circ\s^k(\s^m)=\s^k(\s^n(\s))\circ\s^k(\s^m)=
\\
=\s^k(\s^n(\s)\circ\s^m)=\s^k(\i_n\circ\s^m)=\s^k(\s^{m+1})=(\i_k)^{m+1}.
\end{multline*}

At the inductive step, we observe that:
$$(\i_{k+n})^{\ell+1}\circ(\i_k)^m=\i_{k+n}\circ[(\i_{k+n})^\ell\circ(\i_k)^m]=
\i_{k+n}\circ(\i_k)^{m+\ell}=(\i_k)^{m+\ell+1},$$
where the last equality follows by the base case of the induction.
\end{proof}

\begin{remark}
The composition $\i_{k+n}\circ(\i_k)^m$ is
not defined when $n>m$. Indeed, for every $X\in\VV$,
one has that $\s^{m+k}(X)=((\i_k)^m\circ\s^k)(X)\in\text{range}((\i_k)^m)$, but
one can easily show that if $X\notin\VV_1$
then $\s^{m+k}(X)\notin\VV_{m+k+1}\supseteq\VV_{k+n}=\text{dom}(\i_{k+n})$.
However, that composition is defined if one restricts to elements
of $\VV_{k+n-m}$. Indeed if $a\in\VV_{k+n-m}$
then $(\i_k)^m(a)\in\VV_{k+n}$, and so the image $\i_{k+n}(a)$ is defined.
\end{remark}

\begin{proposition}\label{compositions2bis2}
Let $n>m$. Then for every $\ell$, the restricted compositions to its natural domain
$$((\i_{k+n})^\ell\circ(\i_k)^m)\res\VV_{k+n-m}=
((\i_k)^m\circ(\i_{k+n-m})^\ell)\res\VV_{k+n-m}.$$
\end{proposition}

\begin{proof}
Let us first consider the case when $\ell=1$.
By the Proposition \ref{compositions2bis}, we have:
\begin{multline*}
\i_{k+n}\circ(\i_k)^m=\ \ldots\ =\s^k(\i_n\circ\s^m)=\s^k(\s^m\circ\i_{n-m})=
\\
=\s^k(\s^m\circ\s^{n-m}(\s))=
\s^k(\s^m)\circ\s^{k+n-m}(\s)=(\i_k)^m\circ\i_{k+n-m}.
\end{multline*}

At the inductive step, note that:
\begin{multline*}
(\i_{k+n})^{\ell+1}\circ(\i_k)^m=\i_{k+n}\circ[(\i_{k+n})^\ell\circ(\i_k)^m]=
\\
=\i_{k+n}\circ[(\i_k)^m\circ(\i_{k+n-m})^\ell]=
[\i_{k+n}\circ(\i_k)^m]\circ(\i_{k+n-m})^\ell=
\\
=[(\i_k)^m\circ\i_{k+n-m}]\circ(\i_{k+n-m})^\ell
=(\i_k)^m\circ(\i_{k+n-m})^{\ell+1}.
\end{multline*}

\end{proof}

\smallskip
As seen above, the internal star maps $\i_{n+1}$ and $\i_n$
coincide on all elements of the form $\i_n(x)$.
We now show that there are no other elements on which $\i_{n+1}$ and $\i_n$ coincide.

\begin{proposition}\label{internalmapscommute2}
Let $\xi,\eta\in\s^{n+1}(X)$. Then $\i_{n+1}(\xi)=\i_n(\eta)$ if and only
there exists $x\in\s^n(X)$ such that $\xi=\eta=\i_n(x)$.
\end{proposition}

\begin{proof}
By Proposition \ref{star=dast}, for any fixed $X\in\VV$, the following is true:
$$\forall \xi,\eta\in\s(X)\ ({}^\circledast\xi={}^*\eta\ \leftrightarrow\  \exists x\in X\ \xi=\eta={}^*x).$$
By using the star map symbol $\s$, we can reformulate as follows:
$$\forall \xi,\eta\in\s(X)\ \ \s(\s\res X)(\xi)=(\s\res X)(\eta)\ \leftrightarrow\
\exists x\in X\ \xi=\eta=(\s\res X)(x).$$
Then by \emph{transfer} applied to the map $\s^n$, we obtain the desired property:
\begin{multline*}
\forall \xi,\eta\in\s^{n+1}(X)\ \ \s^{n+1}(\s\res X)(\xi)=\s^n(\s\res X)(\eta)\ \leftrightarrow\
\\
\leftrightarrow\ \exists x\in\s^n(X)\ \xi=\eta=\s^n(\s\res X)(x).
\end{multline*}
\end{proof}

\emph{E.g.}, when $n=1$ we have that for all
$\xi,\eta\in\VV_2=\bigcup_{X\in\VV}{}^{**}\!X$, one has
$\i_2(\xi)=\i_1(\eta)={}^\circledast\eta$ if and only if
$\xi=\eta={}^\circledast x$ for some $x\in\VV_1=\bigcup_{X\in\VV}\hX$.

\begin{remark}\label{internalpullback}
The previous proposition states that the following
commutative diagram is the \emph{pullback} in the category of sets
of the embeddings $\i_{n+1}:\VV_{n+1}\to\VV_{n+2}$
and $\i_n\res\VV_{n+1}:\VV_{n+1}\to\VV_{n+2}$:
$$\begin{tikzcd}
\VV_n\arrow{r}{\i_n}
\arrow{d}{\i_n} &
\VV_{n+1}\arrow{d}{\i_{n+1}} \\
\VV_{n+1}\arrow{r}{\i_n\,\res\,\VV_{n+1}} & \VV_{n+2}
\end{tikzcd}$$
\end{remark}

\smallskip
Let us now see a few examples to illustrate the use of
the equalities proved above.

Observe that the equality $\i_n\circ\s^n=\s^{n+1}$ means that
the $n$-th internal star map $\i_n$ coincides with the
star map $\s$ on those elements that are $n$-th iterated star images,
\emph{i.e.} elements of the form $\s^n(x)$ for some $x\in\VV$.
Thus for every $x\in\VV$ one has that:

\begin{itemize}
\item
$\i_1({}^*x)={}^{\circledast*}x={}^{**}x$;\footnote
{~As already seen in Proposition \ref{double}.}
\item
$\i_2({}^{**}x)=\i_1({}^{**}x)={}^{***}x$;
\item
$\i_3({}^{***}x)=\i_2({}^{***}x)=\i_1({}^{***}x)={}^{****}x$, and so forth.
\end{itemize}

Similarly, the equality $\i_{k+1}\circ\i_k=\i_k\circ\i_k$
means that the internal star maps $\i_{k+1}$ and $\i_k$
coincide on all elements of the form $\i_k(x)$. \emph{E.g.}, when $k=1$
one has that $\i_2({}^\circledast x)={}^{\circledast\circledast}x$ for every $x\in\VV_1$.

More generally, when $n\le m$, the equality $\i_{k+n}\circ(\i_k)^m=(\i_k)^{m+1}$ means that
internal star maps $\i_{k+n}$ and $\i_k$ coincides on all
elements that are $m$-th iterated $k$-th internal star images,
\emph{i.e.} elements of the form $(\i_k)^m(x)$ for some $x\in\VV_k$.
\emph{E.g.}, when $k=1$, for every $x\in\VV_1$ one has:

\begin{itemize}
\item
$\i_3({}^\circledast x)=\i_2({}^\circledast x)={}^{\circledast\circledast}x$;
\item
$\i_3({}^{\circledast\circledast}x)={}^{\circledast\circledast\circledast}x$;
\item
$\i_4({}^\circledast x)={}^{\circledast\circledast}x$;
\item
$\i_4({}^{\circledast\circledast}x)={}^{\circledast\circledast\circledast}x$;
\item
$\i_4({}^{\circledast\circledast\circledast}x)={}^{\circledast\circledast\circledast\circledast}x$, and so forth.
\end{itemize}

\subsection{Nonstandard natural and real numbers of higher levels}

\

\noindent
The hypernatural numbers $\hN$ and the hyperreal numbers $\hR$
are fundamental objects of nonstandard analysis.
In our context, we can also consider ``higher level" nonstandard extensions
of the natural and real numbers.

\begin{definition}
Let $\N_0=\N$ and $\R_0=\R$; and, for $n\in\N$,
let $\N_n=\s^n(\N)$ and $\R_n=\s^n(\R)$.

We will refer to $\N_n$ and $\R_n$ as the nonstandard natural numbers
and nonstandard real numbers of \emph{level} $n$, respectively.

\end{definition}


\smallskip
It is readily seen that $\N_n,\R_n\in\VV_n$ belong
to the $n$-th internal universe, where we agree that $\VV_0=\VV$.
However, both $\N_n$ and $\R_n$ are ``external" objects with
respect to the internal universes $\VV_m$ where $m>n$.

\begin{proposition}
For every $n<m$, the sets $\N_n,\R_n\notin\VV_m$.
\end{proposition}

\begin{proof}
It is a basic fact in nonstandard analysis that $\N$ and $\R$ are
``external" sets, \emph{i.e.}, they do not belong to the (first)
internal universe $\VV_1$.
This means that $\N,\R\notin\hX$ for any $X\in\VV$. Then,
by \emph{transfer} applied to the iterated star map $\s^n$,
one obtains that $\N_n,\R_n\notin\s^n(\hX)=\s^{n+1}(X)$ for any $X\in\VV$, and
hence $\N_n,\R_n\notin\VV_{n+1}$.
Finally, note that if $m>n$ then $\VV_m\subseteq\VV_{n+1}$.
\end{proof}


Starting from the proper inclusions $\N_0=\N\subset\hN=\N_1$
and $\R_0=\R\subset\hR=\R_1$, by \emph{transfer} we obtain the following increasing chains:
$$\N_0\subset\N_1\subset\cdots\subset\N_n\subset\N_{n+1}\subset\cdots$$
$$\R_0\subset\R_1\subset\cdots\subset\R_n\subset\R_{n+1}\subset\cdots$$
It is interesting to note that the direction of the inclusions above is opposite of
the direction in (2) of Proposition \ref{structure1}.



\begin{proposition}
For all $n>m$, the ordered set $\N_n$ is an end-extension of $\N_m$, i.e.,
$\N_m\subset\N_n$ and $x<y$ for all
$x\in\N_m$ and $y\in\N_n\setminus\N_m$.
\end{proposition}

\begin{proof}
It is enough to show that $\N_{m+1}$ is an end-extension of $\N_m$
for every $m\in\N$. Indeed, the binary relation of ``end-extension" is a transitive one.

Consider the true sentence: ``$\hN$ is an end-extension of $\N$",
and apply \emph{transfer} with respect to the bounded elementary embedding $\s^m$.
Then we obtain that $\N_{m+1}$ is an end-extension of $\N_m$.
\end{proof}

\begin{corollary}
${}^*\xi>\eta$ for every $\xi\in\N_n\setminus\N_{n-1}$ and for every $\eta\in\N_n$.
\end{corollary}

\begin{proof}
Note that if $\xi\in\N_n\setminus\N_{n-1}$ then ${}^*\xi\in\N_{n+1}\setminus\N_n$.
Then apply the property that $\N_{n+1}$ is an end-extension of $\N_n$.
\end{proof}

\smallskip
For instance, we have that ${}^*\xi>\eta$ for every $\xi\in\hN\setminus\N$ and
for every $\eta\in\hN$; in particular ${}^*\xi>\xi$ for every $\xi\in\hN\setminus\N$.

\begin{proposition}
\

\begin{enumerate}
\item
$\i_n(\xi)=\xi$ for every $\xi\in\N_n$.
\item $\i_n(\xi)\in\N_{n+m+2}\setminus\N_{n+m+1}$ for every
$\xi\in\N_{n+m+1}\setminus\N_{n+m}$.
\item
$\i_n(\xi)>\eta$ for every $\xi\in\N_{n+m+1}\setminus\N_{n+m}$
and every $\eta\in\N_{n+m+1}$.
\end{enumerate}
\end{proposition}

\begin{proof}
(1).\ We know that ${}^*n=n$ for every $n\in\N$.
Then, by applying \emph{transfer} to the sentence:
``$\forall x\in\N\ (\s\res\N)(x)=x$" with respect to the map $\s^n$, we obtain
that ``$\forall \xi\in\s^n(\N)\ \s^n(\s\res\N)(\xi)=\xi$", as desired.

\smallskip
(2).\ Apply \emph{transfer} with respect to the map $\s^n$ to the following sentence:
$$\forall x\in\N_{m+1}\setminus\N_{m}\,\left(\s(x)\in\N_{m+2}\setminus\N_{m+1}\right).$$
We obtain
that ``$\forall \xi\in\N_{n+m+1}\setminus\N_{n+m}\,
\left(\i_n(\xi)\in\N_{n+m+2}\setminus\N_{n+m+1}\right)$", as desired.

\smallskip
(3).\ It is an easy consequence of (2).
\end{proof}

\smallskip
For instance, we have that ${}^\circledast\xi>\eta$
for every $\xi\in{}^{**}\N\setminus\hN$ and
for every $\eta\in{}^{**}\N$; in particular ${}^\circledast\xi>\xi$ for every $\xi\in{}^{**}\N\setminus\hN$.
We then have that
$\i_2(\xi)>\eta$ for every $\xi\in{}^{**}\N\setminus\hN$ and
for every $\eta\in{}^{**}\N$; and so forth.

\begin{remark}
If $\xi,\eta\in\N_n\setminus\N_{n-1}$ for some $n\geq 2$ then both inequalities
${}^*\xi<{}^\circledast\eta$ and ${}^*\xi>{}^\circledast\eta$ are possible.
\emph{E.g.}, let $\xi={}^*a$ and $\eta={}^*b$
be elements of $\N_2\setminus\N_1$ where $a,b\in\hN\setminus\N$.
Then ${}^*\xi={}^{**}a$ and ${}^\circledast\eta={}^{**}b$ and so,
by \emph{transfer},
we have that ${}^*\xi<{}^\circledast\eta$ if and only if $a<b$.

Similarly, one can show that for $\xi,\eta\in\N_3\setminus\N_2$,
all possible (compatible) inequalities can hold between the pairs
${}^*\xi$ and ${}^\circledast\eta$, ${}^\circledast\xi$ and $\i_2(\eta)$,
${}^*\xi$ and $\i_2(\eta)$; and more generally, for all $n<m$, for
pairs ${}^*\xi$ and $\i_n(\eta)$, $\i_n(\xi)$ and $\i_m(\eta)$,
${}^*\xi$ and $\i_m(\eta)$.
\end{remark}

\subsection{Adding external objects to embeddings}
\

\noindent
In all the above we have considered bounded elementary
embedding between transitive sets, considered
as structures in the language of set theory where the
membership relation is interpreted as the true membership.

As we have seen above, the sets $\N_n,\R_n\notin\VV_m$ when $n<m$.
However, since $\N_n\subset\R_n\subset\R_m\subset\VV_m$,
one can add those ``external" sets to the structure $(\VV_m,\in)$
as suitable interpretations of unary relation symbols, and maintain
the property of bounded elementary embedding.

To clarify this, let us see one example.
(Other examples can be obtained along the same lines).

\begin{proposition}
Let $\mathcal{L}=\{\in,\mathbf{N}_0,\mathbf{R}_0,
\mathbf{N}_1,\mathbf{R}_1,\mathbf{N}_2,\mathbf{R}_2\}$
be the first-order language of set theory augmented with
four unary relation symbols, and consider the following $\mathcal{L}$-structures:
\begin{itemize}
\item
$(\VV_3;\N,\R,\hN,\hR,{}^{**}\N,{}^{**}\R)$
where the membership relation symbol is interpreted as the true membership,
and $\mathbf{N}_0,\mathbf{R}_0,\mathbf{N}_1,\mathbf{R}_1,\mathbf{N}_2,\mathbf{R}_2$
are interpreted as $\N,\R,\hN,\hR,{}^{**}\N,{}^{**}\R$, respectively.
\item
$(\VV_4;\hN,\hR,{}^{**}\N,{}^{**}\R,{}^{***}\N,{}^{***}\R)$
where the membership relation symbol is interpreted as the true membership
and $\mathbf{N}_0,\mathbf{R}_0,\mathbf{N}_1,\mathbf{R}_1,\mathbf{N}_2,\mathbf{R}_2$
are interpred as $\hN,\hR,{}^{**}\N,{}^{**}\R,{}^{***}\N,{}^{***}\R$, respectively.
\end{itemize}

Then the restriction
$$\s\res\VV_3:(\VV_3;\N,\R,\hN,\hR,{}^{**}\N,{}^{**}\R)\prec_b
(\VV_4;\hN,\hR,{}^{**}\N,{}^{**}\R,{}^{***}\N,{}^{***}\R)$$
is a bounded elementary embedding.
Similarly, also the restriction
$$\i_1\res\VV_3:(\VV_3;\hN,\hR,{}^{**}\N,{}^{**}\R)\prec_b
(\VV_4;{}^{**}\N,{}^{**}\R,{}^{***}\N,{}^{***}\R)$$
is a bounded elementary embedding.
%
%
%
%
%
%
%
%
%
%
%
%
\end{proposition}

\begin{proof}
For simplicity, let us call
$\mathfrak{V}_3=(\VV_3;\N,\R,\hN,\hR,{}^{**}\N,{}^{**}\R)$
and $\mathfrak{V}_4=(\VV_4;\hN,\hR,{}^{**}\N,{}^{**}\R,{}^{***}\N,{}^{***}\R)$.
All sets $\N,\R,\hN,\hR,{}^{**}\N,{}^{**}\R\in\VV$
belong to the domain of the star map $\s$.
Since $\s\res\VV_3:\VV_3\prec_b\VV_4$ is a bounded elementary embedding,
and $\s$ maps the interpretation of the symbols $\mathbf{N}_i$ and $\mathbf{R}_i$
in $\mathfrak{V}_3$ to the corresponding interpretations in $\mathfrak{V}_4$,
it easily follows that $\s\res\VV_3:\mathfrak{V}_3\prec_b\mathfrak{V}_4$
is in fact a bounded elementary embedding.

\smallskip
Now consider the reduced structures
$\mathfrak{V}'_3=(\VV_3;\hN,\hR,{}^{**}\N,{}^{**}\R)$
and $\mathfrak{V}'_4=(\VV_4;{}^{**}\N,{}^{**}\R,{}^{***}\N,{}^{***}\R)$
to the sublanguage
 $\mathcal{L}'=\!\{\in,\mathbf{N}_1,\mathbf{R}_1,\mathbf{N}_2,\mathbf{R}_2\}$.
Note that the sets $\hN,\hR,{}^{**}\N,{}^{**}\R\in\VV_1$
belong to the domain of $\i_1$.
Note also that $\i_1$ maps the interpretation $\s^{i}(\N)$ of the symbol
$\mathbf{N}_i$ in $\mathfrak{V}'_3$ to the corresponding
interpretation $\s^{i+1}(\N)=\i_1(\s^i(\N))$ in $\mathfrak{V}'_4$,
and similarly for the symbols $\mathbf{R}_i$.
As in the previous case, since $\i_1\res\VV_3:\VV_3\prec_b\VV_4$
is a bounded elementary embedding,
it easily follows that $\i_1\res\VV_3:\mathfrak{V}'_3\prec_b\mathfrak{V}'_4$
is in fact a bounded elementary embedding.
\end{proof}

\subsection{A normal form for the canonical embeddings}
\

\noindent
By putting together the results proved above, one obtains
the existence of a ``normal form" for all possible
canonical embeddings between the universes $\VV$ and $\VV_n$,
as given by compositions of the internal star maps with the star map.

\begin{theorem}[Normal Form]\label{normalform}
Every composition of internal star maps with the star map
equals a unique composition of the following form:
$$\s^{\ell_0}\circ(\i_{n_1})^{\ell_1}\circ\ \ldots\ \circ(\i_{n_k})^{\ell_k}$$
on the domain $\VV_{n_k}$
where $0<n_1<\ldots<n_k$ and
$\ell_i\not=0$ for $i=1,2,\ldots,k$. We allow $\ell_0=0$
and $k=0$, but not both.\footnote
{~We agree that $\s^0=\imath$ is the identity map; and we agree
that when $k=0$ the above composition equals $\s^{\ell_0}$.}
\end{theorem}

\begin{proof}
For convenience, let us temporarily denote $\s=\i_0$.
For all $i,j\ge 0$ define:
$$\sigma(\i_i,\i_j)=\begin{cases}
0\ & \text{if}\ i\le j ;
\\
i-j\ & \text{if}\  i>j.
\end{cases}$$

Now let $\psi_1\circ\psi_2\circ\,\ldots\,\circ\psi_N$ be any
composition where the functions $\psi_i$ are either the star map $\s$
or one of the internal star maps $\i_n$.
Define the \emph{rank} $\rho$ of the composition by letting
$$\rho(\psi_1\circ\psi_2\circ\,\ldots\,\circ\psi_N)=\sum_{i=1}^{N-1}\sigma(\psi_i,\psi_{i+1}).$$
If $N=1$, we agree that the rank of a single function $\rho(\psi)=0$.

We proceed by induction on the rank of the finite composition.
Observe that if $\rho(\psi_1\circ\psi_2\circ\,\ldots\,\circ\psi_N)=0$,
then either we have a single function,
or $\rho(\psi_i,\psi_{i+1})=0$ for every $i=1,\ldots,N-1$.
In both cases, it is clear that one can write
$\psi_1\circ\psi_2\circ\,\ldots\,\circ\psi_N$ in the desired canonical normal form.
This proves the base case for all compositions of rank $0$.

At the inductive step when $\rho(\psi_1\circ\psi_2\circ\,\ldots\,\circ\psi_N)>0$,
pick $i$ such that $\sigma(\psi_i,\psi_{i+1})\geq 1$.
We have two possibilities: either $\psi_i=\i_k$ for some $k$ and $\psi_{i+1}=\s$,
or $\psi_i=\i_k$ and $\psi_{i+1}=\i_m$ where $k>m$.

In the former case, $\psi_i\circ\psi_{i+1}=\i_k\circ\s=\s\circ\i_{k-1}$,
by Proposition \ref{compositions2bis}.
Note that $\sigma(\s,\i_{k-1})=0$ and so,
by interchanging $\psi_i$ and $\psi_{i+1}$
we obtain an equal composition with a smaller rank, and hence
the inductive hypothesis applies. In the latter case, we have
$\psi_i\circ\psi_{i+1}=\i_k\circ\i_m=\i_m\circ\i_{k-1}$,
by Proposition \ref{compositions2bis2}, and similarly as above,
by interchanging $\psi_i$ and $\psi_{i+1}$,
we obtain an equal composition to which the inductive hypothesis applies.

Let us now prove the uniqueness of the normal form.
Suppose that a composition $\Psi$ has two normal forms
$$\Psi=\s^{\ell_0}\circ(\i_{n_1})^{\ell_1}\circ\ \ldots\ \circ(\i_{n_k})^{\ell_k}
=\s^{\ell'_0}\circ(\i_{m_1})^{\ell'_1}\circ\ \ldots\ \circ(\i_{m_{k'}})^{\ell'_{k'}}.$$
Note that the domain of $\Psi$ is $\VV_{n_k}$ as well as $\VV_{m_{k'}}$, and hence $n_k=m_{k'}$.
If $a\in\N_{n_k+1}\setminus\N_{n_k}=\N_{m_{k'}+1}\setminus\N_{m_{k'}}$,
then $\Psi(a)\in\N_{n_k+1+\ell_k+\cdots+\ell_0}\setminus\N_{n_k+\ell_k+\cdots+\ell_0}$ and
$\Psi(a)\in\N_{n_k+1+\ell'_{k'}+\cdots+\ell'_0}
\setminus\N_{n_k+\ell'_{k'}+\cdots+\ell'_0}$; therefore $L=\ell_0+\cdots+\ell_k=\ell'_0+
\cdots+\ell'_{k'}$.

For every $a\in\N_{n_k}\setminus\N_{n_k-1}=\N_{m_{k'}}\setminus
\N_{m_{k'}-1}$ one has $\i_{n_k}(a)=a$, and so
\begin{multline*}
\Psi(a)=\s^{\ell'_0}\circ(\i_{m_1})^{\ell'_1}\circ\
\ldots\ \circ(\i_{m_{k'-1}})^{\ell'_{k'-1}}(a)\\
\in\N_{m_{k'}+L-\ell'_{k'}}\setminus\N_{m_{k'}-1+L-\ell'_{k'}}=
\N_{n_k+L-\ell'_{k'}}\setminus\N_{n_k-1+L-\ell'_{k'}}\,\mbox{ as well as}
\end{multline*}
$$\Psi(a)=\s^{\ell_0}\circ(\i_{n_1})^{\ell_1}\circ\ \ldots\
\circ(\i_{n_{k-1}})^{\ell_{k-1}}(a)
\in\N_{n_k+L-\ell_{k}}\setminus\N_{n_k-1+L-\ell_{k}}.$$
It follows that $\ell_k=\ell'_{k'}$.

By induction, suppose we have proved that
$n_{k-i}=m_{k'-i}$ and $\ell_{k-i}=\ell'_{k'-i}$ for $i=0,1,\ldots,j-1$ for some $j\leq k$.
If $n_{k-j}<m_{k'-j}$, pick an element $a\in\N_{n_{k-j}+1}\setminus\N_{n_{k-j}}$. Then
$\i_{n_{k-j}}(a)\in \N_{n_{k-j}+2}\setminus\N_{n_{k-j}-1}$ and
$\i_{m_{k'-j}}(a)=a$, and so
$$\Psi(a)\in\N_{n_{k-j}+1+L-\ell_{k-j+1}-\,\ldots\,-\ell_{k}}\setminus
\N_{n_{k-j}+L-\ell_{k-j+1}-\,\ldots\,-\ell_k}\,\mbox{ as well as}$$
$$\Psi(a)\in\N_{n_{k-j}+1+L-\ell'_{k'-j}-\ell'_{k'-j+1}-\,\ldots\,-\ell'_{k'}}\setminus
\N_{n_{k-j}+L-\ell'_{k'-j}-\ell'_{k'-j+1}-\,\ldots\,-\ell'_{k'}}.$$
This implies that $\ell'_{k'-j}=0$, contradicting the assumption. By symmetry,
one can also show that $n_{k-j}>m_{k'-j}$ is impossible, and hence $n_{k-j}=m_{k'-j}$.
By picking an element $a\in\N_{n_{k-j}}\setminus\N_{n_{k-j}-1}$, it can be shown by
the same argument as above that $\ell_{k-j}=\ell'_{k'-j}$.
This completes the proof of uniqueness.
\end{proof}

\section{Various bounded elementary embeddings}\label{sec-bee}

In this section we will study various canonical
ways of obtaining bounded elementary embeddings
between internal universes. To this end, a useful
notion will be that of transitive closure.
Recall the following definitions.

\begin{definition}
A set $T$ is \emph{transitive} if elements of elements of $T$ are elements of $T$:
$$\forall t\in T\ \forall x\in t\ \ x\in T.$$
The \emph{transitive closure} $\text{TC}(A)$ of a set $A$ is the smallest
transitive set that includes $A$.
\end{definition}

It is easily seen that the transitive closure $\text{TC}(A)=\bigcup_{n\ge 0}A_n$ where
$A_0=A$ and, inductively, $A_{n+1}=\bigcup\{a\mid a\in A_n\}$.

A fundamental property that we will use in the sequel is the fact that
bounded quantifier formulas are absolute between transitive sets
(see, \emph{e.g.} \cite[Ch.\,12]{je}).

\begin{itemize}
\item
Let $T\subset T'$ be transitive sets. Then for every
bounded quantifier formula $\varphi(x_1,\ldots,x_n)$ and
for all $t_1,\ldots,t_n\in T$, one has:
$$T\models \varphi(t_1,\ldots,t_n)\ \Longleftrightarrow\
T'\models\varphi(t_1,\ldots,t_n).$$
\end{itemize}

Recall that our universe is given by the superstructure $\VV=\bigcup_{k\ge 0}V_k(\XX)$,
which is a transitive set. Note that every level $V_k(\XX)\in\VV$.

\smallskip
The following properties about internal universes hold.

\begin{proposition}
Let $n\in\N$. Then:

\begin{enumerate}
\item
$\VV_n=\bigcup_{k\ge 0}\s^n(V_k(\XX))$.
\item
$\VV_n$ is the transitive closure of $\text{range}(\s^n)$.
\item
$\VV_{n+1}\subset\VV_n$, and hence the internal universes form a (strictly) decreasing chain:
$\VV\supset\VV_1\supset\VV_2\supset\ \cdots\ \supset\VV_n\supset\VV_{n+1}\supset\ \cdots$.
\end{enumerate}
\end{proposition}

\begin{proof}
(1).\ One inclusion is trivial because every
$V_k(\XX)\in\VV$.
For the converse inclusion $\bigcup_{X\in\VV}\s^n(X)\subseteq\bigcup_{k\ge 0}\s^n(V_k(\XX))$,
observe that for every $X\in\VV$ we can pick $k$ such that $X\subseteq V_k(\XX)$,
and hence $\s^n(X)\subseteq\s^n(V_k(\XX))$.

\smallskip
(2).\ For every $k$, the property:
``$\forall x\in V_k(\XX)\ \forall y\in x\ y\in V_{k-1}(\XX)$" holds in $\VV$.
By applying \emph{transfer} with respect to $\s^n$,
and by taking (1) into account, we obtain that $\VV_n$ is transitive.
Now observe that, by definition, $\VV_n$ contains elements of elements of
the range of $\s^n$,
and hence it is included in the transitive closure of $\text{range}(\s^n)$.
So, we are left to show that $\text{range}(\s^n)\subseteq\VV_n$.
To see this, note that every $X\in\VV$ belongs to $V_k(\XX)$ for some $k$,
and hence $\s^n(X)\in\s^n(V_k(\XX))\subseteq\VV_n$.

\smallskip
(3).\ For every $X\in\VV$, we have that $\hX\subseteq V_k(\XX)$ for some $k$
and so, by \emph{transfer} with respect to $\s^n$, we
obtain that $\s^{n+1}(X)\subseteq\s^n(V_k(\XX))\subset\VV_n$.
The inclusion $\VV_{n+1}\subset\VV_n$  is proper because, \emph{e.g.},
$\N\in V_1(\XX)$ and $\N\notin{}^*V_k(\XX)$ for every $k$ implies that
$\s^n(\N)\in\s^n(V_1(\XX))\subset\VV_n$ but $\s^n(\N)\notin\bigcup_{k\ge 0}\s^{n+1}(V_k(\XX))$.
\end{proof}

\begin{proposition}
For every $n$, $\VV_{n+1}$ is the transitive closure of $\text{range}(\i_n)$.
\end{proposition}

\begin{proof}
As checked right after the definition of $\i_n$, we know that $\text{range}(\i_n)\subseteq\VV_{n+1}$;
besides, we have already seen that every universe $\VV_{n+1}$ is transitive,
and hence the transitive closure of $\text{range}(\i_n)$ is included in $\VV_{n+1}$.
On the other hand, by Proposition \ref{compositions2} we have that
$\s^{n+1}(X)=\i_n(\s^n(X))\in\text{range}(\i_n)$
for every $X\in\VV$, and so we also have the converse inclusion
$\VV_{n+1}\subseteq\text{range}(\i_n)$.
\end{proof}

\smallskip
Let us now consider restrictions of iterated star maps and
iterated internal star maps.

\begin{proposition}\label{rangeofrestrictions}
Let $m,n,\ell\in\N$. Then:
\begin{enumerate}
\item
The transitive closure of $\text{range}(\s^m\res\VV_\ell)$ is $\VV_{m+\ell}$.
\item
The transitive closure of $\text{range}((\i_n)^m\res \VV_{n+\ell})$ is $\VV_{n+m+\ell}$.
\end{enumerate}
\end{proposition}

%


\begin{proof}
(1).\ For every $X\in\VV$, one has that
$\s^{m+\ell}(X)=\s^m(\s^\ell(X))$\newline
$\in\text{range}(\s^m)$,
and so $\VV_{m+\ell}=\bigcup_{X\in\VV}\s^{m+\ell}(X)\subseteq\text{TC}(\text{range}(\s^m))$.
Conversely, if $a\in\s^\ell(X)$, then $\s^m(a)\in\s^{m+\ell}(X)$, and so
$\text{range}(\s^m\res\VV_\ell)\subseteq\VV_{m+\ell}$.
Since $\VV_{m+\ell}$ is transitive, we conclude that
$\text{TC}(\text{range}(\s^m\res\VV_\ell))=\VV_{m+\ell}$.

\smallskip
(2).\ Let $X\in\VV$ be any set.
If $Y=\s^\ell(X)$ then $\s^m\res Y:\s^\ell(X)\to\s^{m+\ell}(X)$,
and $\s^n(\s^m\res Y):\s^{n+\ell}(X)\to\s^{n+m+\ell}(X)$.
So, if $a\in\s^{n+\ell}(X)$ then
$(\i_n)^m(a)=\s^n(\s^m\res Y)(a)\in\s^{n+m+\ell}(X)\subseteq\VV_{n+m+\ell}$.
This shows that $\text{range}((\i_n)^m\res \VV_{n+\ell})\subseteq\VV_{n+m+\ell}$.
Since $\VV_{n+m+\ell}$ is transitive, we are left to show
that $\VV_{n+m+\ell}\subseteq\text{TC}(\text{range}((\i_n)^m\res \VV_{n+\ell}))$.
To see this, recall that $(\i_n)^m\circ\s^{n+\ell}=\s^{n+m+\ell}$;
then for every $X\in\VV$ we have that
$\s^{n+m+\ell}(X)=((\i_n)^m\circ\s^{n+\ell})(X)=(\i_n)^m(\s^{n+\ell}(X))\in\text{range}((\i_n)^m)$,
and the desired inclusion follows.
\end{proof}

\subsection{Bounded elementary chains}\label{subsec-bee}
\

\noindent
The \emph{transfer principle} for the star map states that $\s:\VV\to\VV$
is a \emph{bounded elementary embedding}, \emph{i.e.}
it satisfies the property of an elementary embedding
restricted to bounded quantifier formulas.
Following \cite[\S 4.4]{ck}, in the sequel we will use the symbol $\prec_b$
for bounded elementary extensions. So, we write:
$$\s:\VV\prec_b\VV$$

Now recall the basic fact that bounded quantifier formulas are absolute
between transitive models (see, \emph{e.g.}, \cite[Ch. 12]{je}).
This means that if $T\subset T'$ are transitive,
then for every bounded quantifier formula $\varphi(x_1,\ldots,x_n)$
and for all $t_1,\ldots,t_n\in T$, one has the equivalence:
$$T\models\varphi(t_1,\ldots,t_n)\Leftrightarrow T'\models\varphi(t_1,\ldots,t_n).$$

Since all universes $\VV$ and $\VV_n$ are transitive,
one obtains in a natural way several examples of bounded elementary embeddings between them.

To begin with, one has the following reversed chain of bounded elementary extensions,
where the underlying maps are the inclusions $\imath$:
$$\VV\stackrel{\imath}{\succ_b}\VV_1\stackrel{\imath}{\succ_b}
\VV_2\stackrel{\imath}{\succ_b}\ \ldots\ \stackrel{\imath}{\succ_b}
\VV_n\stackrel{\imath}{\succ_b}\VV_{n+1}\stackrel{\imath}{\succ_b}\ \ldots$$


\smallskip
Since the range of the star map is included in the transitive universe
$\VV_1=\bigcup_{X\in\VV}\hX$ (which is indeed its transitive closure),
it follows that $\s$ is still
a bounded elementary embedding when viewed as taking values in $\VV_1$:
$$\s:\VV\prec_b\VV_1.$$

For every $\ell$, we have also seen that the transitive closure of
$\text{range}(\s^m\res\VV_\ell)$ is $\VV_{m+\ell}$. So, the restriction
$\s^m\res\VV_\ell:\VV_\ell\prec_b\VV_{m+\ell}$, and we have the following
chain of bounded elementary extensions where all the underlying
maps are (restrictions of) the star map:
$$\VV\stackrel{*}{\prec_b}\VV_1\stackrel{*}{\prec_b}
\VV_2\stackrel{*}{\prec_b}\ \ldots\ \stackrel{*}{\prec_b}
\VV_n\stackrel{*}{\prec_b}\VV_{n+1}\stackrel{*}{\prec_b}\ \ldots$$

As we have seen, the \emph{transfer principle} also holds for all
$n$-th internal star maps $\i_n$ (Theorem \ref{n-dualtransfer}),
\emph{i.e.} $\i_n:\VV_n\prec_b\VV$.
Since the range of the $n$-th internal star map $\i_n$ is included
in the transitive universe $\VV_{n+1}$ (which is indeed its transitive closure),
similarly as observed above for the star map, we have the following chain
of bounded elementary embeddings, where the underlying maps are the internal star maps:
$$\VV_1\stackrel{\i_1}{\prec_b}\VV_2\stackrel{\i_2}{\prec_b}
\VV_3\stackrel{\i_3}{\prec_b}\ \ldots\ \prec_b
\VV_n\stackrel{\i_n}{\prec_b}\VV_{n+1}\stackrel{\i_{n+1}}{\prec_b}\VV_{n+2}\prec_b \ldots$$

We remark that plenty of other examples can be obtained, starting from the
observation that for all $n\le m$, one has that
$\VV_m\subseteq\VV_n$,
and both $\text{range}(\s\res\VV_m)$ and
$\text{range}(\i_n\res\VV_m)$ are included in $\VV_{m+1}$.
For instance, for all non-negative integers
$\ell,n_0,n_1,\ldots,n_k$
where $n_i\le\ell+i$, one has the following chain of bounded elementary embeddings:
$$\VV_\ell\stackrel{\i_{n_0}}{\prec_b}\VV_{\ell+1}\stackrel{\i_{n_1}}{\prec_b}
\VV_{\ell+2}\stackrel{\i_{n_2}}{\prec_b}\ \ldots\ \prec_b
\VV_{\ell+k}\stackrel{\i_{n_k}}{\prec_b}\VV_{\ell+k+1},$$
where we agree that $\VV_0=\VV$ and that $\i_0=\s$.

Then
for every $\ell$ and for every increasing sequence $(n_k)$
where $\ell\ge n_1$ and $n_i\leq\ell+i$,
we have the following chain of bounded elementary embeddings:
$$\VV_\ell\stackrel{\i_{n_1}}{\prec_b}\VV_{\ell+1}\stackrel{\i_{n_2}}{\prec_b}
\VV_{\ell+2}\stackrel{\i_{n_3}}{\prec_b}\ \ldots\ \prec_b
\VV_{\ell+k}\stackrel{\i_{n_k}}{\prec_b}\VV_{\ell+k+1}\stackrel{\i_{n_k+1}}{\prec_b}\ \ldots$$



\subsection{First examples of canonical embeddings}

\
%
%


\noindent
There are two different canonical ways of embedding the first internal
universe $\VV_1$ into the second internal universe $\VV_2$, namely:
\begin{itemize}
\item
$\s\res\VV_1:\VV_1\prec_b\VV_2$;
\item
$\i_1:\VV_1\prec_b\VV_2$.
\end{itemize}


There are three different canonical ways of embedding the
second internal universe $\VV_2$ into the third internal universe $\VV_3$, namely:

\begin{itemize}
\item
$\s\res\VV_2:\VV_2\prec_b\VV_3$;
\item
$\i_1\res\VV_2:\VV_2\prec_b\VV_3$;
\item
$\i_2:\VV_2\prec_b\VV_3$.
\end{itemize}

In general, for every $n$, one has the following $(n+1)$-many
pairwise different canonical
embedding from the $n$-th internal universe $\VV_n$ into
the $(n+1)$-th internal universe $\VV_{n+1}$, namely:

\begin{itemize}
\item
$\s\res\VV_n:\VV_n\prec_b\VV_{n+1}$;
\item
$\i_1\res\VV_n:\VV_n\prec_b\VV_{n+1}$;
\item
\ldots
\item
$\i_{n-1}\res\VV_n:\VV_n\prec_b\VV_{n+1}$;
\item
$\i_n:\VV_n\prec_b\VV_{n+1}$.
\end{itemize}

Let us now turn to the more general case of maps from $\VV_n$ to $\VV_m$ where $n<m$,
which are obtained by composing the embeddings considered above.

We have just seen the $m=n+1$, so let us turn to the case when $m\ge n+2$.

As an example, let us consider $n=1$ and $m=3$.
As seen above, we have two maps from $\VV_1$ to $\VV_2$, namely
$(\s\res\VV_1)$ and $\i_1$, and three maps from $\VV_2$ to $\VV_3$, namely
$(\s\res\VV_2)$, $(\i_1\res\VV_2)$, and $\i_2$.\footnote
{~For simplicity, in the following we will not specify to what universes
the various maps are restricted, as this will be clear from the context.}

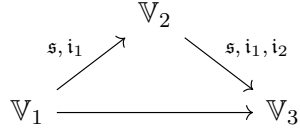
$$\begin{tikzcd}
& \VV_2 \arrow{dr}{\s,\,\i_1,\,\i_2} \\
\VV_1 \arrow{ur}{\s,\,\i_1} \arrow{rr}{} && \VV_3
\end{tikzcd}$$

So, we obtain six possible compositions giving rise to maps from
$\VV_1$ into $\VV_3$, namely:
$$(1)\ \s\circ\s\ ;\ (2)\ \i_1\circ\s\ ;\ (3)\ \i_2\circ\s\ ;\ (4)\ \s\circ\i_1\ ;\ (5)\ \i_1\circ\i_1\ ;\ (6)\ \i_2\circ\i_1\ .$$

However, not all the above compositions are different to each other.
Indeed, we have that $(1)=(2)=**$, $(3)=(4)=*\circledast$, and $(5)=(6)=\circledast\circledast$
(see Propositions \ref{double}, \ref{compositions2}, and \ref{compositions3}, respectively).

We remark that the three embeddings $**$, $*\circledast$, and $\circledast\circledast$
are different to each other, as a particular case of the Normal Form Theorem \ref{normalform}.

\section{Combinatorics with internal and external star maps}\label{sec-combinatorics}

In this section we present a few basic results that connect
the star map and the internal star maps with combinatorial properties.\footnote{
~Other results along this line can be found in \cite{dn3}.}

\begin{theorem}\label{Ramsey1}
Let $I$ be an infinite set, let
$\alpha,\beta\in{}^*\!I\setminus\{{}^*i\mid i\in I\}$ and let $X\subseteq I\times I$.
Then the following are true.
\begin{enumerate}
\item
If $({}^\circledast\alpha,{}^*\beta)\in{}^{**}\!X$ then
there exist 1-1 sequences $(a_n)_{n\in\N}$
and $(b_m)_{m\in\N}$ such that $(a_n,b_m)\in X$ for all $n\le m$.
\item
If $({}^\circledast\alpha,{}^*\alpha)\in{}^{**}\!X$ then
there exist a 1-1 sequence $(h_n)_{n\in\N}$
such that $(h_n,h_m)\in X$ for all $n<m$.
\end{enumerate}
\end{theorem}

\begin{proof}
Note that if $A=\{i_1,\ldots,i_k\}\subset I$ is finite,
then $\alpha\notin\hA$, since $\hA=\{{}^*i_1,\ldots,{}^*i_k\}$.
So, our hypothesis $\alpha\in{}^*\!I\setminus\{{}^*i\mid i\in I\}$
guarantees the following: ``If $\alpha\in\hA$ for a subset $A\subseteq I$ then $A$ is infinite".
Of course, the same properties also hold for $\beta$.

\smallskip
$(1)$. By the hypothesis we know that $({}^*(\s\res I)(\alpha),{}^*\beta)\in{}^{**}\!X$, and so
$\alpha\in{}^*\Lambda$ where $\Lambda=\{a\in I\mid ((\s\res I)(a),\beta)\in\hX\}$.
Pick $a_1\in\Lambda$; then $(\s\res I(a_1),\beta)=({}^*a_1,\beta)\in\hX$,
and hence $\beta\in{}^*(X_{a_1})$,
where $X_{a_1}=\{b\in I\mid (a_1,b)\}$ is the vertical fiber of $X$ at $a_1$.
Now pick any $b_1\in X_{a_1}$; then $(a_1,b_1)\in X$.
Since $\Lambda$ is infinite, we can pick an element $a_2\ne a_1$ in $\Lambda$;
then $\beta\in{}^*(X_{a_2})$. Since $\beta\in{}^*(X_{a_1}\cap X_{a_2})$,
we can pick an element $b_2\ne b_1$ in $X_{a_1}\cap X_{a_2}$;
then $(a_1,b_2),(a_2,b_2)\in X$.
Proceed inductively in this way,
and at step $n$ pick $a_{n+1}\ne a_1,\ldots,a_n$
in $\Lambda$, and $b_{n+1}\ne b_1,\ldots,b_n$ in $\bigcap_{s=1}^{n+1} X_{a_s}$.
This is possible because $\beta\in\bigcap_{s=1}^{n+1}\hX_{a_s}={}^*(\bigcap_{s=1}^{n+1}X_{a_s})$.
Then $(a_n)_{n\in\N}$ and $(b_m)_{m\in\N}$ are infinite 1-1 sequences with the desired property.

\smallskip
$(2)$. The construction is similar to the one above.
By the hypothesis, $\alpha\in{}^*\Gamma$
where $\Gamma=\{a\in I\mid ((\s\res I)(a),\alpha)\in\hX\}$, and so we can pick
$h_1\in\Gamma$; then $\alpha\in{}^*(X_{h_1})$.
Then proceed inductively by picking, at step $n$,
an element $h_{n+1}\ne h_1,\ldots,h_n$
in $\Gamma\cap\bigcap_{s=1}^n X_{h_s}$. This is possible because
$\alpha\in{}^*\Gamma\cap\bigcap_{s=1}^n{}^*(X_{h_s})=
{}^*(\Gamma\cap\bigcap_{s=1}^n X_{h_s})$.
It is readily seen that the 1-1 sequence $(h_n)_{n\in\N}$ satisfies the required property.
\end{proof}

Let us now focus on the particular case when $I=\N$.
Recall that ${}^*n=n$ for every $n\in\N$,
and hence ${}^\circledast\alpha=\alpha$ for every $\alpha\in\hN$.
Note also that when $I=\N$, in the inductive step in the
proof of the previous Theorem, one can pick $a_{n+1}>a_n$ and $b_{n+1}>b_n$,
and hence we can assume that both sequences of natural numbers
$(a_n)_{n\in\N}$ and $(b_n)_{n\in\N}$ are increasing.
So, when $I=\N$, the previous theorem states the following.

\begin{corollary}\label{halfsquaresN}
Let $\alpha,\beta\in\hN\setminus\N$,
and let $X\subseteq\N\times\N$. Then
\begin{enumerate}
\item
If $(\alpha,{}^*\beta)\in{}^{**}\!X$ then
there exist increasing sequences $(a_n)_{n\in\N}$
and $(b_m)_{m\in\N}$ such that $(a_n,b_m)\in X$ for all $n\le m$.
\item
If $(\alpha,{}^*\alpha)\in{}^{**}\!X$ then
there exist an increasing sequence $(h_n)_{n\in\N}$
such that $(h_n,h_m)\in X$ for all $n<m$.
\end{enumerate}
\end{corollary}

\smallskip
Note that the above corollary yields a direct proof of Ramsey's Theorem for pairs.
(In the next \S\ref{sec-applications} we will also see
a proof of the general case for $k$-tuples).

\begin{theorem}[Ramsey for pairs]
Let $[\N]^2=C_1\cup\ldots\cup C_r$ be a finite coloring of the pairs of natural numbers.
Then there exists a color $C_i$ and an infinite set $H\subseteq\N$ such that
the pairs in $[H]^2=\{(h,h')\in H\times H\mid h<h'\}\subseteq C_i$ are monochromatic.
\end{theorem}

\begin{proof}
For convenience, let us identify $[\N]^2$ with
the upper diagonal $\Delta^+=\{(a,b)\in\N\times\N\mid a<b\}$.
Note that, by \emph{transfer},
${}^{**}([\N]^2)=[{}^{**}\N]^2={}^{**}C_1\cup\ldots\cup{}^{**}C_r$
is a finite coloring of the pairs of ${}^{**}\N$.
Now pick any $\alpha\in\hN\setminus\N$, and recall that $\alpha<{}^*\alpha$.
Then the pair $(\alpha,{}^*\alpha)\in{}^{**}C_i$ for some $i$.
Finally, by the previous corollary,
there exists an increasing sequence $H=(h_n)$ such
that $[H]^2=\{(h_i,h_j)\mid i<j\}\subseteq C_i$, as desired.
\end{proof}

\subsection{The $u$-equivalence}
\

\noindent
The following equivalence relation on nonstandard natural numbers
revealed useful in studying partition regularity problems (see, \emph{e.g.}, \cite{dl}).

\begin{definition}
(\cite[\S 2]{dn2}) We say that two $k$-tuples $\vec{x},\vec{y}\in\hN^k$ are \emph{$u$-equivalent}, and
write $\vec{x}\ueq\vec{y}$, if $\vec{x}\in\hX\Leftrightarrow \vec{y}\in\hX$
for every $X\subseteq\N^k$.

More generally, if $\vec{x}\in(\N_n)^k$ and $\vec{y}\in(\N_m)^k$, we put
$\vec{x}\ueq\vec{y}$ if $\vec{x}\in\s^n(X)\Leftrightarrow\vec{y}\in\s^m(X)$
for every $X\subseteq\N^k$.
\end{definition}

The name ``$u$-equivalence" was given because
two elements $\vec{x},\vec{y}\in\hN^k$ are $u$-equivalent
if and only if the corresponding generated ultrafilters on $\N^k$ are equal:
$$\UU_{\vec{x}}:=\{A\subseteq\N^k\mid \vec{x}\in\hA\}=
\{A\subseteq\N^k\mid \vec{y}\in\hA\}=\UU_{\vec{y}}.$$

In the next proposition we itemize a few basic properties of $u$-equivalence
that we will use in the sequel.

\begin{proposition}\label{ueqproperties}
\

\begin{enumerate}
\item[$(i)$]
If $(x_1,\ldots,x_k)\,\ueq\,(y_1,\ldots,y_k)$ then $x_i\,\ueq\,y_i$ for every $i=1,\ldots,k$.\item[$(ii)$]
If $\vec{x}\,\ueq\,\vec{y}$ then $\s^n(\vec{x})\,\ueq\,\s^n(\vec{y})$ for every $n\in\N$.
\item[$(iii)$]
Let $\vec{x},\vec{y}\in(\N_n)^k$ and $f:\N^k\to\N^h$. If
$\vec{x}\,\ueq\,\vec{y}$, then \newline
$\s^n(f)(\vec{x})\,\ueq\,\s^n(f)(\vec{y})$.
\end{enumerate}
\end{proposition}

\begin{proof}
$(i)$: Let $n$ be sufficiently large such that $x_i,y_i\in\N_n$ for $i=1,2,\ldots,k$.
Note that $\vec{x}\in\s^m(X)$ if and only if $\vec{x}\in\s^{m+1}(X)$ for any $X\subseteq\N^k$
as long as $\vec{x}\in\N_m^k$.

Given an arbitrary $X_i\subseteq\N$, let $X_j=\N$ for $j\not=i$
and $X=\prod_{j=1}^kX_j$. Note that $\s^n(X)=\prod_{j=1}^k\s^n(X_j)$.
Then, $x_i\in \s^n(X_i)$ if and only if $(x_1,\ldots,x_k)\in\s^n(X)$ if and only if
$(y_1,\ldots,y_k)\in\s^n(X)$ if and only if $y_i\in\s^n(X_i)$.

$(ii)$: Let $m\in\N$ be such that $\vec{x},\vec{y}\in\N^k_m$ and $n\in\N$.
For any $X\subseteq\N^k$, we have $\s^n(\vec{x})\in\s^{m+n}(X)$ if and only if
$\vec{x}\in\s^m(X)$ if and only if $\vec{y}\in\s^m(X)$ if and only if $\s^n(\vec{y})
\in\s^{m+n}(X)$.

$(iii)$: Let an arbitrary $X\in\N^h$ be given. Then, $\s^n(f)(\vec{x})\in\s^n(X)$
if and only if $\vec{x}\in\s^n(f)^{-1}(\s^n(X))=\s^n(f^{-1}(X))$ if and only if
$\vec{y}\in\s^n(f^{-1}(X))$ if and only if $\s^n(f)(\vec{y})\in\s^n(X)$.
%
%
\end{proof}

\begin{remark}
The converse implication in property $(i)$ above does not hold.
\end{remark}

However, under suitable hypotheses, one can preserve $u$-equivalence
when passing to ordered tuples.

\begin{proposition}
Let $x_i,y_i\in\N_\ell$ be such that $x_i\ueq y_i$
for every $i=1,\ldots,k$.
Then for all $0\le n_1<\ldots<n_k$ and $0\le m_1<\ldots<m_k$
where $n_{i+1}\ge n_i+\ell$ and $m_{i+1}\ge m_i+\ell$, we have
$(\s^{n_1}(x_1),\ldots,\s^{n_k}(x_k))\ueq(\s^{m_1}(y_1),\ldots,\s^{m_k}(y_k))$.\footnote
{~Precisely, if $x_i\ueq y_i$ generate the ultrafilter
$\U_i=\UU_{x_i}=\UU_{y_i}$ for $i=1,\ldots,k$, then both
$(\s^{n_1}(x_1),\ldots,\s^{n_k}(x_k))$
and $(\s^{m_1}(y_1),\ldots,\s^{m_k}(y_k))$ generate the tensor product
$\U_1\otimes\cdots\otimes\U_k$.}
\end{proposition}

\begin{proof}
We proceed by induction on $k$. The base case $k=1$ is trivial.
At the inductive step, note that for every $X\subseteq\N^k$, one has
$(\s^{n_1}(x_1),\ldots,\s^{n_k}(x_k))\in\s^{n_k+\ell}(X)\Leftrightarrow
(x_1,\s^{n_2-n_1}(x_2),\ldots,\s^{n_k-n_1}(x_k))\in\s^{n_k-n_1+\ell}(X)\Leftrightarrow
x_1\in\s^\ell(A)$ where
$$A:=\{t\in\N\mid (t,\s^{n_2-n_1-\ell}(x_2),\ldots,\s^{n_k-n_1-\ell}(x_k))\in
\s^{n_k-n_1}(X)\}.$$
In the same way, for every $X\subseteq\N^k$ one has that
$(\s^{m_1}(y_1),\ldots,\s^{m_k}(y_k))\in\s^{m_k+\ell}(X)\Leftrightarrow
y_1\in\s^\ell(B)$ where
$$B:=\{t\in\N\mid (t,\s^{m_2-m_1-\ell}(y_2),\ldots,\s^{m_k-m_1-\ell}(x_k))\in
\s^{m_k-m_1}(X)\}.$$
Now let $n'_i:=n_i-n_1-\ell$ and $m'_i:=m_i-m_1-\ell$ for $i=2,\ldots,k$.
Clearly, the sequences $0\le n'_2<\ldots<n'_k$ and $0\le m'_2<\ldots<m'_k$
are such that $n'_{i+1}\ge n'_i+\ell$ and $m'_{i+1}\ge m'_i+\ell$, and so we can apply
the inductive hypothesis to get
$(\s^{n'_2}(x_2),\ldots,\s^{n'_k}(x_k))\ueq(\s^{m'_2}(y_2),\ldots,\s^{m'_k}(y_k))$.
It follows that for every $t\in\N$ one has
$(t,\s^{n_2-n_1-\ell}(x_2),\ldots,\s^{n_k-n_1-\ell}(x_k))\in
\s^{n_k-n_1}(X)\Leftrightarrow (t,\s^{m_2-m_1-\ell}(x_2),\ldots,\s^{m_k-m_1-\ell}(x_k))
\in\s^{m_k-m_1}(X)$,
and hence $A=B$. Finally, since $x_1\ueq y_1$, we have
the equivalence $x_1\in\s^\ell(A)\Leftrightarrow y_1\in\s^\ell(B)$, and the proof is completed.
\end{proof}

\begin{corollary}\label{ueqmultiple}
If $\alpha,\alpha',\beta,\beta'\in\hN$ are such that
$\alpha\ueq\alpha'$ and $\beta\ueq\beta'$,
then $(\alpha,{}^*\!\beta)\ueq(\alpha',{}^*\!\beta')$, and hence
${}^{**}f(\alpha,{}^*\!\beta)\ueq{}^{**}f(\alpha',{}^*\!\beta')$
for every $f:\N\times\N\to\N$.
\end{corollary}

\begin{proof}
One obtains $(\alpha,{}^*\!\beta)\ueq(\alpha',{}^*\!\beta')$
from the previous proposition where $x_1=\alpha$, $x_2=\beta$,
$y_1=\alpha'$, $y_2=\alpha'$, $\ell=1$,
$n_1=m_1=0$, and $n_2=m_2=1$.
Then, ${}^{**}f(\alpha,{}^*\!\beta)\ueq{}^{**}f(\alpha',{}^*\!\beta')$
directly follows from Proposition \ref{ueqproperties} $(iii)$.
\end{proof}

The relevance of the $u$-equivalence relation as a tool in Ramsey problems is grounded
on the following characterization.\footnote
{~The particular case of partition regularity of Diophantine equations
where one considers families of the form $\G_P=\{(x_1,\ldots,x_n )\mid P(x_1,\ldots,x_n)=0\}$
where $P\in\Z[X]$ can be found in \cite[\S 3]{dn1}.}

\begin{proposition}\label{PRequivalence}
Let $\G\subseteq\PP(\N)$ be a family of finite patterns.
Then the following are equivalent:
\begin{enumerate}
\item
$\G$ is partition regular on $\N$, \emph{i.e.} for every finite coloring
of $\N$ there exists a monochromatic $G\in\G$.
\item
There exists $X\in{}^*\G$ (or, equivalently, there exists $X\in\s^n(\G)$ for some $n$)
such that all elements in $X$ are $u$-equivalent to each other.
\end{enumerate}
\end{proposition}

\begin{proof}
To avoid triviality we can assume that
each member in $\G$ has at least two elements.

\smallskip
$(1)\Rightarrow(2)$.\
Denote $F^c:=\N\setminus F$ and let
\[\U_0:=\{F\subseteq\N\mid F^c\,\mbox{ does not contain any element in }\,\G\}.\]
Then, $\U_0$ has finite intersection property. Indeed, if
$\bigcap_{i=1}^kF_i=\emptyset$ for some $F_i\in\U_0$, then
$\N=\bigcup_{i=1}^kF^c_i$. Note that $\{F^c_1,F^c_2\setminus F^c_1,\ldots,
F^c_k\setminus\bigcup_{i=1}^{k-1}F^c_i\}$ is a partition of $\N$ and, clearly,
no cell in the partition contains a member of $\G$. Note also that
any finite set is not in $\U_0$ because if a finite set $F_0\in\U_0$,
then $F^c_0$ and $\{x\}$ for each $x\in F_0$ form a finite partition of $\N$
and no cell in the partition contains members of $\G$.

Let $\U$ be an ultrafilter containing $\U_0$. Then every $F\in\U$ contains
some members from $\G$ because otherwise $F^c$ would be in $\U_0\subseteq\U$.
By $\mathfrak{c}^+$-enlargement property we conclude that the intersection of
all $^*\!F$ for $F\in\U$ contains a member $X$ in $^*\G$. Clearly,
all elements in $X$ belongs to the same $u$-equivalence class.

\smallskip
$(2)\Rightarrow(1)$.\ Let $X\in{}^*\G$ be such that $x\ueq y$ for all
$x,y\in X$. Fix $x_0\in X$. For any partition
$\{C_1,C_2,\ldots,C_k\}$ of $\N$
there is an $i\leq k$ such that $x_0\in{}^*C_i$
which implies that $X\subseteq{}^*C_i$. By {\it transfer}
we can find a $Y\in\G$ such that $Y\subseteq C_i$.
\end{proof}

For instance, recall \emph{Schur's Theorem}:
``In every finite coloring of $\N$ there exists a monochromatic triple $a, b, a+b$ with $a\ne b$."
Then the partition regularity property stated
in Schur's Theorem is equivalent to the existence of
numbers $\alpha\ne\beta$ in $\hN$ (or in $\s^n(\N)$ for some $n$)
such that $\alpha\,\ueq\,\beta\,\ueq\,\alpha+\beta$.

\smallskip
In the sequel, we will use the following combinatorial property
of $u$-equivalence with respect to iterated star maps.

\begin{proposition}\label{starofpairs}
Let $\alpha,\beta\in\hN\setminus\N$ where $\alpha\ueq\beta$ and $\alpha\ne\beta$,
and let $X\subseteq\N^4$.
If $(\alpha,{}^*\alpha,{}^*\beta,{}^{**}\beta)\in{}^{***}\!X$ then
there exist an increasing sequence $(a_n)_{n\in\N}$ such that
$(a_i,a_{2j},a_{2j+1},a_k)\in X$ for all $i<2j<2j+1<k$.
\end{proposition}

\begin{proof}
We will use the following notation for $a,a',a''\in\N$:
\begin{itemize}
\item
$\Lambda:=\{n\in\N\mid(n,\alpha,\beta,{}^*\beta)\in{}^{**}\!X\}$;
\item
$\Gamma_a:=\{(m,s)\in\N\times\N\mid a,m,s\ \text{are distinct and}\ (a,m,s,\beta)\in\hX\}$;
\item
$X_{a,a'\!,a''}=\{t\in\N\mid (a,a',a'',t)\in X\}$.
\end{itemize}
By the hypothesis, $\alpha\in{}^*\!\Lambda$, and so we can pick $a_1\in\Lambda$.
Then $(a_1,\alpha,\beta,{}^*\beta)\in{}^{**}\!X$, and hence $(\alpha,\beta)\in{}^*\Gamma_{a_1}$.
Since $\alpha\ueq\beta$, also $\beta\in{}^*\!\Lambda$. Note that
$a_1,\alpha,\beta$ are distinct, and so
$(\alpha,\beta)\in {}^*\Gamma_{a_1}\cap{}^*(\Lambda\times\Lambda)$.
Pick $(a_2,a_3)\in\Gamma_{a_1}\cap(\Lambda\times\Lambda)$,
so that $a_1, a_2, a_3$ are distinct, $(a_1,a_2,a_3,\beta)\in\hX$,
and $a_2,a_3\in\Lambda$. We have that $\beta\in\hX_{a_1,a_2,a_3}$,
and hence also $\alpha\in{}^*\!X_{a_1,a_2,a_3}$. Then
$$(\alpha,\beta)\in{}^*(X_{a_1,a_2,a_3}\times X_{a_1,a_2,a_3})\cap
{}^*\Gamma_{a_1}\cap{}^*\Gamma_{a_2}\cap{}^*\Gamma_{a_3}\cap{}^*(\Lambda\times\Lambda),$$
and we can pick $(a_4,a_5)\in (X_{a_1,a_2,a_3}\times X_{a_1,a_2,a_3})\cap
\Gamma_{a_1}\cap\Gamma_{a_2}\cap\Gamma_{a_3}\cap(\Lambda\times\Lambda)$.
This means that $a_1,a_2,a_3,a_4,a_5$ are distinct,
$(a_1,a_2,a_3,a_4),(a_1,a_2,a_3,a_5)\in X$;
$(a_i,a_4,a_5,\beta)\in\hX$, and hence $\beta\in{}^*\!X_{a_i,a_4,a_5}$, for $i=1,2,3$;
and $a_4,a_5\in\Lambda$, and hence $(\alpha,\beta)\in{}^*\Gamma_{a_4}$ and
$(\alpha,\beta)\in{}^*\Gamma_{a_5}$. By inductively iterating this process,
one defines the desired 1-1 sequence $(a_n)_{n\in\N}$.

Precisely, assume that pairwise different elements $a_1,\ldots,a_{2j+1}$ have been defined
in such a way that:
\begin{itemize}
\item[$(a)$]
$(a_i,a_{2s},a_{2s+1},a_t)\in X$ for all $i<2s<2s+1<t\le 2j+1$;
\item[$(b)$]
$\alpha,\beta\in\hX_{a_i,a_{2s},a_{2s+1}}$ for all $i<2s<2j$;
\item[$(c)$]
$(\alpha,\beta)\in{}^*\Gamma_{a_i}$ for every $i\le 2j+1$.
\end{itemize}
Then $(\alpha,\beta)$ belongs to ${}^*\!\mathcal{X}$ where
$$\mathcal{X}=\bigcap_{i<2s\le 2j}(X_{a_i,a_{2s},a_{2s+1}}\times X_{a_i,a_{2s},a_{2s+1}})\cap
\bigcap_{i=1}^{2j+1}\Gamma_{a_i}\cap(\Lambda\times\Lambda),$$
and we can pick $(a_{2j+2},a_{2j+3})\in\mathcal{X}$.
Then it is verified in a straightforward manner that
elements $a_1,\ldots,a_{2j+1},a_{2j+2},a_{2j+3}$ are pairwise different
and still satisfy the required conditions $(a), (b), (c)$.
\end{proof}

\section{Applications}\label{sec-applications}

We present here three examples of applications where we use
the nonstandard technique of iterated star maps and internal star maps.\footnote
{~More applications will be found in \cite{dr}.}

\subsection{Ramsey's Theorem}
\

\noindent
We present first a nonstandard proof of the classic Ramsey's Theorem
in the general case for $k$-tuples. The proof has perhaps
the simplest combinatorial argument.

\begin{theorem}
Let $k,r\in\N$ be positive. For every coloring $c:[\N]^k\to[r]$
there exists an infinite set $H\subseteq\N$ such that
$[H]^k$ is monochromatic.
\end{theorem}

\begin{proof}
The case $k=1$ is trivial, so let us assume $k\ge 2$.
To begin with, we observe that for every $i$, the iterated star extension
$\s^i(c):[\N_i]^k\to[r]$ is an $r$-colorings of the $k$-tuples of $\N_i$
that extends $c$, \emph{i.e.} $\s^i(c)\res[\N]^k=c$.

We work in the $k$-th internal universe $\VV_k$.
Pick any $\alpha\in\N_1\setminus\N$, and consider the
$k$-tuple
$$\overline{\alpha}:=\{\alpha<\s(\alpha)<\ldots<\s^{k-1}(\alpha)\}\in[\N_k]^k.$$
Let $t=\s^k(c)(\overline{\alpha})\in [r]$ be the color of the $k$-tuple $\overline{\alpha}$.
By induction, we define an increasing sequence $(a_m)$ of natural numbers
in such a way that, for every $m\ge 0$, all $k$-tuples from
$\{a_1,\ldots,a_m\}\cup\overline{\alpha}$
have color $t$, \emph{i.e.} $\s^k(c)\res[\{a_1,\ldots,a_m\}\cup\overline{\alpha}]^k\equiv t$.


The base case $m=0$ is trivial since we assumed $\s^k(\overline{\alpha})=t$, so
let us consider the inductive step $m+1$.

The element $y=\alpha$ witnesses that the following sentence is true:
$$\exists y\in\N_1\,(y>a_m\,\wedge\,
\s^k(c)\res[\{a_1,\ldots,a_m\}\cup\{y,\s(\alpha),\s^2(\alpha),\ldots,\s^{k-1}(\alpha)\}]^k
\equiv t).$$
By \emph{backward transfer} we obtain that:
\[\exists y\in\N\,(y>a_m\,\wedge\,
\s^{k-1}(c)\res[\{a_1,\ldots,a_m\}\cup\{y,\alpha,\s(\alpha),\ldots,\s^{k-2}(\alpha)\}]^k
\equiv t).\]
So, we can pick a natural number $a_{m+1}>a_m$ such that
\begin{equation}\label{indhypothesis}
\s^{k-1}(c)\res[\{a_1,\ldots,a_m,a_{m+1}\}\cup\{\alpha,\s(\alpha),\ldots,\s^{k-2}(\alpha)\}]^k\equiv t.
\tag{$\dagger$}
\end{equation}
We want to show that $\s^k(c)\res[\{a_1,\ldots,a_m,a_{m+1}\}\cup\overline{\alpha}]^k\equiv t$.

Let $\overline{b}=\{b_1<b_2<\cdots<b_k\}\in[\{a_1,\ldots,a_m,a_{m+1}\}\cup\overline{\alpha}]^k$.
If $\overline{b}=\overline{\alpha}$, then trivially $\s^k(c)(\overline{b})=t$.
If $b_k<\s^{k-1}(\alpha)$, then
$\s^k(c)(\overline{b})=\s^{k-1}(c)(\overline{b})=t$ by (\ref{indhypothesis}).

We are left to consider the case when $b_k=\s^{k-1}(\alpha)$ and $\overline{b}\not=\overline{\alpha}$;
then there exists $\s^p(\alpha)\notin\overline{b}$ where $0\le p\le k-2$.
Define the $k$-tuple $\overline{b}'=\{b'_1,\ldots,b'_k\}$ as follows:
If $b_i<\s^p(\alpha)$ let $b'_i=b_i$;
if $b_i=\s^q(\alpha)>\s^p(\alpha)$ then let $b'_i=\s^{q-1}(\alpha)$.
Now recall that the $p$-th internal star map is such that
$\i_p(\s^\ell(\alpha))=\s^{\ell+1}(\alpha)$ whenever $\ell\ge p$;
recall also that $\i_p$ is the identity on all elements of $\N_p$.
Thus we have $\i_p(b'_i)=b_i$ for every $i=1,\ldots,k$, and hence
$\i_p(\overline{b}')=\overline{b}$.
By property (\ref{indhypothesis}), we know that $\s^{k-1}(c)(\overline{b}')=t$,
and so we can conclude that
$\s^k(c)(\overline{b})=\i_p(\s^{k-1}(c))(\i_p(\overline{b}'))=
\i_p(\s^{k-1}(c)(\overline{b}'))=\i_p(t)=t$.
\end{proof}

\begin{remark}
(1) Different from some traditional proof, 
our proof above does not do induction on the length of tuple $k$.
(2) Although constructing a nonstandard universe often needs the existence of a
non-principal ultrafilter on an infinite set, the proof above can actually be
carried out without the axiom of choice by recent work of K.\ Hrb\'{a}\v{c}ek \cite{hr2}.
\end{remark}

\subsection{Multidimensional van der Waerden's Theorem}
\

\noindent
In this subsection, we give a nonstandard proof of the multidimensional
van der Waerden's Theorem. The relevant point here is that we only use
a simple induction; this is to be contrasted with the usual proof that uses
nested inductive arguments.

The multidimensional van der Waerden's Theorem
was first proved by Gallai, and indeed it is also known as Gallai's Theorem.
Our nonstandard proof presented below was inspired by the proof of the
one-dimensional van der Waerden's Theorem given in \cite{khinchin}.

%

\begin{theorem}[T.\ Gallai]\label{vdw}
Let $s\in\N$. For every finite coloring of $\N^s$ and for every finite $S\subset\N^s$
there exist $\vec{a}\in\N^s$ and $d\in\N$
such that the homothetic copy $\vec{a}+d\cdot S=\{\vec{a}+d\,\vec{x}\mid x\in S\}$
is monochromatic.
\end{theorem}

\begin{proof}
Fix any enumeration $\N^s=(\vec{b}_n\mid n\in\N)$ of the countable set $\N^s$.
We will prove Gallai's Theorem by showing that
the following property $\varphi(\ell)$ holds for every $\ell\in\N$.

\begin{itemize}
\item
$\varphi(\ell)$:\ For every finite coloring of $\N^s$ there exist $\vec{a}\in\N^s$
and $d\in\N$ such that vectors
$\vec{a}, \vec{a}+d\cdot\vec{b}_1, \ldots, \vec{a}+d\cdot\vec{b}_\ell$
are monochromatic.
\end{itemize}

\smallskip
We proceed by induction. The base case $\ell=1$ is trivial.
At the inductive step, assume $\varphi(\ell-1)$.
Then, by Proposition \ref{PRequivalence}, there exist $\vec{\alpha}\in\hN^s$ and $\delta\in\hN$ such that
$$\vec{\alpha}\,\ueq\,\vec{\alpha}+\delta\cdot\vec{b_1}\,\ueq\,\ldots\,\ueq\,
\vec{\alpha}+\delta\cdot\vec{b}_{\ell-1}.$$

Let $\N^s=C_1\cup\ldots\cup C_r$ be a given a finite coloring.
We will work in the $(r+1)$-th internal universe $\VV_{r+1}$,
where $r$ is the number of colors. Consider the following $r+1$ elements in $(\N_{r+1})^s$:
\begin{itemize}
\item
$\vec{\xi}_t=\sum_{j=0}^{t-1}\s^j(\vec{\alpha}+\delta\cdot\vec{b}_\ell)+\sum_{j=t}^r \s^j(\vec{a})$
for $t=0,1,\ldots,r$.
\end{itemize}

Note that $(\N_{r+1})^s=\s^{r+1}(C_1)\cup\ldots\cup\s^{r+1}(C_r)$ is a finite coloring and so,
by the pigeonhole principle, there exist two indexes $t<v$ such that
$\vec{\xi}_t,\vec{\xi}_v\in\s^{r+1}(C_i)$ are monochromatic.
Since $\vec{\alpha}\,\ueq\,\vec{\alpha}+\delta\cdot\vec{b_i}$
are $u$-equivalent elements in $\hN^s$ for every $i=1,\ldots\ell-1$, it follows that
$$\s^j(\vec{\alpha})\,\ueq\,\s^j(\vec{\alpha}+\delta\cdot\vec{b_i})=
\s^j(\vec{\alpha})+\s^j(\delta)\cdot\vec{b}_i$$
are $u$-equivalent elements in $(\N_{j+1})^s$, and hence
\begin{multline*}
\vec{\xi}_t\,=\,\sum_{j=0}^{t-1}\s^j(\vec{\alpha}+\delta\cdot\vec{b}_\ell)+
\sum_{j=t}^{v-1}\s^j(\vec{\alpha})+\sum_{j=v}^{r}\s^j(\vec{\alpha})\,\ueq\,
\\
\ueq\,\sum_{j=0}^{t-1}\s^j(\vec{\alpha}+\delta\cdot\vec{b}_\ell)+
\sum_{j=t}^{v-1}\s^j(\vec{\alpha})+\sum_{j=t}^{v-1}\s^j(\delta)\cdot\vec{b}_i+
\sum_{j=v}^{r}\s^j(\vec{\alpha}).
\end{multline*}
If we let $\vec{\beta}:=\vec{\xi}_t$ and $\varepsilon:=\sum_{j=t}^{v-1}\s^j(\delta)$,
this means that
$$\vec{\beta}\,\ueq\,\vec{\beta}+\varepsilon\cdot\vec{b}_1\,\ueq\, \ldots\,
\ueq\,\vec{\beta}+\varepsilon\cdot\vec{b}_{\ell-1}.$$
Since $\vec{\beta}=\vec{\xi}_t\in\s^{r+1}(C_i)$,
all the above $\ell$ elements belong to $\s^{r+1}(C_i)$.
Now we observe that
$$\vec{\beta}+\varepsilon\cdot\vec{b}_{\ell}=
\vec{\xi}_t+\sum_{j=t}^{v-1}\s^j(\delta)\cdot\vec{b}_\ell=\vec{\xi}_v\in\s^{r+1}(C_i),$$
and hence the following property holds:
\begin{itemize}
\item
There exist $\vec{\beta}\in(\N_{r+1})^s$ and $\varepsilon\in\N_{r+1}$ such that
$$\vec{\beta}, \vec{\beta}+\varepsilon\cdot\vec{b}_1,\, \ldots\,,
\vec{\beta}+\varepsilon\cdot\vec{b}_{\ell-1},
\vec{\beta}+\varepsilon\cdot\vec{b}_{\ell}\in\s^{r+1}(C_i).$$
\end{itemize}
Finally, by backward \emph{transfer} applied to the iterated star map $\s^{r+1}$,
we obtain the desired property $\varphi(\ell)$.
\end{proof}

\subsection{Monochromatic exponential triples}
\

\noindent
We now show how to use the nonstandard technique of iterated
star maps to give a short nonstandard proof of a result of
Sahasrabudhe \cite{sa} about the existence of monochromatic exponential patterns.
In the paper \cite{dr}, the existence of monochromatic triples
$a, b, b^a$ was given a much shorter proof by using algebra on the space
of ultrafilters $\beta\N$, but the initial idea of the authors was actually based on the use of
iterated star maps. We present here the original nonstandard argument.

The proof is grounded on the following well-known improvement
of \emph{van der Waerden's Theorem}, where also the common difference
belongs to the same color of the arithmetic progression.\footnote
{~See, \emph{e.g.},  \cite[Ch.3\ Thm.\,2]{grs}.}

\begin{itemize}
\item
Brauer's Theorem (1928).
\emph{For every finite coloring $\N=C_1\cup\ldots\cup C_r$
and for every $\ell\in\N$ there exists a color $C_i$ and elements
$$a,d,a+d,\ldots,a+\ell d\in C_i.$$}
\end{itemize}

In terms of $u$-equivalence, the above result can be reformulated as follows:

\begin{proposition}\label{brauer}
There exist $\gamma,\delta\in\hN$ such that for every $\ell\in\N$:
$$\gamma\,\ueq\,\delta\,\ueq\,\gamma+\ell\delta.$$
\end{proposition}

\begin{proof}
Observe that the family $\{X_{A,\ell}\mid A\subseteq\N, \ell\in\N\}$ where
$$X_{\!A,\ell}:=\{(c,d)\in\N\times\N\mid c,d,\ldots,c+\ell d\in A\
\text{or}\ c,d,\ldots,c+\ell d\notin A\}.$$
has the finite intersection property.
Indeed, given $A_1,\ldots,A_k\subseteq\N$ and $\ell_1,\ldots,\ell_k\in\N$,
let $\N=C_1\cup\ldots\cup C_r$ be the finite partition determined
by the sets $\{A_1,\ldots,A_k\}$, and let $\ell=\max\{\ell_1,\ldots,\ell_k\}$.
Then apply \emph{Brauer's Theorem} to find a monochromatic
pattern $c, d, c+d, \ldots, c+\ell d\in C_i$;
it is readily seen that $(c,d)\in\bigcap_{j=1}^k X_{\!A_j,\ell_j}$.
Finally, by the $\mathfrak{c}^+$-enlargement property,
we can pick $(\gamma,\delta)\in\bigcap_{A\subseteq\N, \ell\in\N}\hX_{\!A,\ell}$.
Clearly, such a pair $(\gamma,\delta)\in\hN\times\hN$ satisfies the desired property.
\end{proof}

\begin{theorem}
For every finite coloring $c:\N\to[r]$ there exists
a monochromatic triple $a, b, 2^ab$ and
a monochromatic exponential triple $x, y, y^x$.

More, there exists an infinite 1-1 sequence $(x_n)$
such that all elements $x_n, 2^{x_n}x_{n+1}$ are monochromatic;
and there exists an infinite 1-1 sequence $(y_n)$
such that all elements $y_n, (y_{n+1})^{y_n}$ are monochromatic.
\end{theorem}

\begin{proof}
We will prove the result by showing the existence of
numbers $\alpha,\beta,\xi,\eta$ in the triple nonstandard extension ${}^{***}\N$
with $\alpha\,\ueq\,\beta\,\ueq\,2^\alpha\beta$ and $\xi\,\ueq\,\eta\,\ueq\,\eta^\xi$.

Pick $\gamma,\delta\in\hN$ as given by Proposition \ref{brauer}.
Note that
$${}^*\gamma\,\ueq\,{}^*\delta\,\ueq\,{}^*\gamma+\lambda{}^*\delta\quad
\text{for every}\ \lambda\in\hN.$$
Indeed, given any $A\subseteq\N$, by \emph{transfer}
we obtain that for every $\lambda\in\hN$, one has
${}^*\gamma,{}^*\delta,{}^*\gamma+\lambda{}^*\delta\in{}^{**}A$
or ${}^*\gamma,{}^*\delta,{}^*\gamma+\lambda{}^*\delta\notin{}^{**}A$.
Now consider the number $\alpha:=2^\gamma\cdot{}^*\delta\in{}^{**}\N\setminus\hN$
and its nonstandard extension $\beta:={}^*\alpha\in{}^{***}\N\setminus{}^{**}\N$.
By the definition of $u$-equivalence, we have that $\alpha\,\ueq\,\beta$.
Since ${}^*\gamma+2^\gamma\cdot{}^*\delta\,\ueq\,{}^*\gamma$, by
properties in Proposition \ref{ueqproperties} and Corollary \ref{ueqmultiple} we have
$$2^\alpha\cdot\beta\ =\
2^{{}^*\!\gamma+2^\gamma\cdot{}^*\delta}\cdot{}^{**}\delta\ \ueq\
2^{{}^*\!\gamma}\cdot{}^{**}{\delta}\ =\ {}^*\alpha\ =\ \beta.$$

As for the exponential triple, note that
$\alpha\,\ueq\, \beta\,\ueq\, 2^\alpha\cdot\beta$ implies that
$2^\alpha\,\ueq\, 2^\beta\,\ueq\, 2^{2^\alpha\beta}=(2^\beta)^{2^\alpha}$.
So, if we let $\xi:=2^\alpha$ and $\eta:=2^\beta$, we found elements
$\xi\,\ueq\, \eta\,\ueq\, \eta^\xi$.

\smallskip
By using Proposition \ref{starofpairs}, the above properties can be strengthened as follows.
For every $A\subseteq\N$, let
$$\Gamma_A=\{(n,m,s,t)\in\N\times\N\times\N\times\N\mid 2^n s, 2^m t, 2^{2^n s}(2^m t)\in A\}.$$

Since $\alpha=2^\gamma\cdot{}^*\delta$,
$\beta=2^{{}^*\gamma}\cdot{}^{**}\delta$, and $2^\alpha\cdot\beta$ are $u$-equivalent,
for every coloring $c:\N\to[r]$ there
exists $t\in[r]$ such that $\s^3(c)(\alpha)=\s^3(c)(\beta)=\s^3(c)(2^\alpha\beta)=t$,
and hence $(\gamma,{}^*\gamma,{}^*\delta,{}^{**}\delta)\in{}^{***}\Gamma_A$
where $A:=c^{-1}(\{t\})$.
Then, by Proposition \ref{starofpairs}, there exists a 1-1 sequence $(a_n)$ such that
for all $i<2j<2j+1<k$ one has that $(a_i,a_{2j},a_{2j+1},a_k)\in\Gamma_A$,
\emph{i.e.}, one has the monochromatic triple
$x=2^{a_i}a_{2j+1}, y=2^{a_{2j}}a_k, 2^xy\in A$.
If we let $x_n:=2^{a_{2n}}a_{2n+3}$, then it readily verified that the triple
$x_n, x_{n+1}, 2^{x_n}x_{n+1}\in A$ for every $n\in\N$.

The similar property for exponential patterns is proved in the same fashion,
by starting from the triple $\xi\,\ueq\, \eta\,\ueq\, \eta^\xi$ considered above.
\end{proof}

\section{Questions}\label{question}

In model theory, one constructs the direct limit
$\M_\infty=\varinjlim\M_n$ of an elementary chain
of $\mathcal{L}$-structures (where $\mathcal{L}$ a given first-order language):
\[\M_0\prec\M_1\prec\M_2\prec\cdots\] by simply taking the countable
``union'' of the chain; that is, one uses the embedding $\M_n\prec\M_{n+1}$
to identify $\M_n$ with an elementary submodel of $\M_{n+1}$,
and then takes the union $\M_\infty=\bigcup_{n\ge 0}\M_n$.
Now let us consider the chain of bounded elementary extensions
between nonstandard universes
$$\VV=\VV_0\prec_b\VV_1\prec_b\VV_2\prec_b\cdots,$$
where the considered bounded elementary embeddings $j_n:\VV_n\prec_b\VV_{n+1}$
could be either the restricted star maps $j_n=\s\res\VV_n:\VV_n\to\VV_{n+1}$,
or the internal star maps $j_n=\i_n:\VV_n\to\VV_{n+1}$.
Following the same procedure as above, in both cases one
should construct the direct limits as the ``union'' of the chain.
However, this requires caution, since the chain is decreasing:
\[\VV=\VV_0\supset\VV_1\supset\VV_2\supset\cdots\]
Besides, the intersection $\bigcap_{n\ge 0}\VV_n$ may only consists
of those elements $x$ such that ${}^*x=x$.

Similarly, one may want to consider the inverse limit of the above chain
where the considered bounded elementary embeddings are
given by the inclusions $\imath_n:\VV_{n+1}\hookrightarrow\VV_n$ between transitive models,
by taking the ``intersection" of the chain.
%

\begin{question}
Can one define (possibly, in a simple way) the direct limits
$\varinjlim(\VV_n,\s\res\VV_n)_{n\ge 0}$ and $\varinjlim(\VV_n,\imath_n)_{n\ge 0}$
as transitive subuniverses of $\VV=\VV_0$?
\end{question}

\begin{question}
Can one define (possibly, in a simple way) the inverse limit
$\varprojlim(\VV_n,\imath_n)_{n\geq 0}$ where $\imath_n:\VV_{n+1}\prec_b\VV_n$ are
the inclusion maps, as a transitive subuniverse of $\VV=\VV_0$?
\end{question}

Another question is about Ramsey properties of elements in
iterated hyper-extensions of the natural numbers.

As we have seen in Corollary \ref{halfsquaresN},
if $\alpha\in\hN\setminus\N$ then the pair $(\alpha,{}^*\alpha)$ satisfies the
following Ramsey property:
``For every set $X\subseteq\N\times\N$ such that $(\alpha,{}^*\alpha)\in{}^{**}X$
there exists an infinite $H\subseteq\N$ such that $[H]^2\subseteq X$."
Now consider the case when $\Omega\in{}^{**}\N\setminus\hN$;
it is natural to ask whether the same Ramsey property as above holds
for all sets $X\subseteq\N\times\N$ such that $(\Omega,{}^*\Omega)\in{}^{***}X$.

\begin{question}
Let $\Omega\in{}^{**}\N\setminus\hN$. Does the following property hold?
\begin{itemize}
\item
If $X\subseteq\N\times\N$ is such that $(\Omega,{}^*\Omega)\in{}^{***}\!X$ then
there exist an infinite $H\subseteq\N$ such that $[H]^2\subseteq X$.
\end{itemize}
And if not, does the above property holds for arbitrary large (finite) sets $H$?

Finally, can one have one of the above properties at least in the special case
when $\Omega={}^{**}f(\alpha,{}^*\beta)$ for some $f:\N\times\N\to\N$
and $\alpha,\beta\in\hN\setminus\N$?
\end{question}

\bibliographystyle{amsalpha}

\end{document}